\documentclass[12pt]{amsart}
\usepackage[latin1]{inputenc}
\usepackage{amsmath}
\usepackage{amsfonts}
\usepackage{times}
\usepackage{amssymb}
\usepackage{mathtools}

\usepackage{graphicx}
\usepackage{amsthm}
\usepackage{mathrsfs}
\usepackage{hyperref}
\usepackage{color}
\usepackage{mathtools}
\usepackage{xcolor}
\usepackage[width=8.50in, height=11.00in, left=1in, right=1in, top=1.00in, bottom=1in]{geometry}

\newcommand{\mres}{\mathbin{\vrule height 1.6ex depth 0pt width
		0.13ex\vrule height 0.13ex depth 0pt width 1.3ex}}

\newcommand{\rr}{\mathbb R}

\newcommand{\calB}{\mathcal B}
\newcommand{\calC}{\mathscr C}

\newcommand{\calH}{\mathcal H}
\newcommand{\calG}{\mathcal G}
\newcommand{\calT}{\mathcal T}

\newcommand{\E}{\textrm{Edge}}
\newcommand{\BF}{\textrm{Bridge}_{\textrm{Flat}}}
\newcommand{\BN}{\textrm{Bridge}_{\textrm{Non-flat}}}
\newcommand{\B}{\textrm{Bridge}}
\newcommand{\VF}{\textrm{V}^{\textrm{Flat}}}
\newcommand{\VN}{\textrm{V}^{\textrm{Non-flat}}}

\newcommand{\dist}{\operatorname{dist}}
\newcommand{\diam}{\operatorname{diam}}

\newcommand{\CF}{\textrm{Cores}_{\textrm{Flat}}}

\newcommand{\urlwofont}[1]{\urlstyle{same}\url{#1}}
\newtheorem{definition}{Definition}

\newtheorem{lemma}{Lemma}
\newtheorem{corollary}{Corollary}
\newtheorem{theorem}{Theorem}
\newtheorem{innercustomthm}{Theorem}
\newenvironment{customthm}[1]
{\renewcommand\theinnercustomthm{#1}\innercustomthm}
{\endinnercustomthm}

\newtheorem{problem}{Problem}
\numberwithin{definition}{section}
\numberwithin{proposition}{section}
\numberwithin{remark}{section}
\numberwithin{lemma}{section}
\numberwithin{corollary}{section}
\numberwithin{theorem}{section}

\title{Rectifiability of pointwise doubling measures in Hilbert Space}
\author{Lisa Naples}
\date{February 18, 2020}
\subjclass[2010]{Primary 28A75}
\keywords{$1$-rectifible measures, purely $1$-unrectifiable measures, doubling measures, Analyst's traveling salesman theorem, Jones' beta numbers}

\address{Department of Mathematics\\ University of Connecticut\\ Storrs, CT 06269-1009}
\email{lisa.naples@uconn.edu}

\begin{document}

\begin{abstract}
	In geometric measure theory, there is interest in studying the interaction of measures with rectifiable sets.
	Here, we extend a theorem of Badger and Schul in Euclidean space to characterize rectifiable pointwise doubling measures in Hilbert space.  
	Given a measure $\mu$, we construct a multiresolution family $\calC^\mu$ of windows, and then we use a weighted Jones' function $\hat{J}_2(\mu, x)$ to record how well lines approximate the distribution of mass in each window.  
	We show that when $\mu$ is rectifiable, the mass is sufficiently concentrated around a lines at each scale and that the converse also holds.  
	Additionally, we present an algorithm for the construction of a rectifiable curve using appropriately chosen $\delta$-nets. 
	Throughout, we discuss how to overcome the fact that in infinite dimensional Hilbert space there may be infinitely many $\delta$-separated points, even in a bounded set.
	Finally, we prove a characterization for pointwise doubling measures carried by Lipschitz graphs.
\end{abstract}

\maketitle

\tableofcontents

\section{Introduction}

\subsection{Background}
One goal of geometric measure theory is to understand the global structure of a measure through analysis of local geometric data.  We use the below terminology to formalize this notion.

\begin{definition}
	Let $(\mathbb{X}, \mathcal{M})$ be a measurable space, and let $\mathcal{N}\subset\mathcal{M}$ be a family of measurable sets. We say
	\begin{enumerate}
		\item $\mu$ is carried by $\mathcal{N}$ if there exist countably many $N_i\in\mathcal{N}$ such that 
		$\mu(\mathbb{X}\setminus\bigcup_{i}N_i)=0$;
		\item $\mu$ is singular to $\mathcal{N}$ if $\mu(N)=0$ for every $N\in\mathcal{N}$.
	\end{enumerate}
\end{definition}

A $\sigma$-finite measure $\mu$ on $(\mathbb{X}, \mathcal{M})$ can be decomposed uniquely as 
\[\mu=\mu_\mathcal{N}+\mu_{\mathcal{N}}^\perp\] where $\mu_{\mathcal{N}}$ is carried by $\mathcal{N}$ and $\mu_\mathcal{N}^\perp$ is singular to $\mathcal{N}$.  In \cite{B19}, Badger poses the following problem:

\begin{problem}[Identification Problem] Let $(\mathbb{X}, \mathcal{M})$ be a measurable space, let $\mathcal{N}\subset\mathcal{M}$ be a family of measurable sets, and let $\mathscr{F}$ be a family of $\sigma$-finite measures defined on $\mathcal{M}$.  Find properties $P(\mu,x)$ and $Q(\mu, x)$ defined for all $\mu\in\mathcal{F}$ and $x\in\mathbb{X}$ such that 
	\[\mu_\mathcal{N}=\mu\mres\left\{x\in\mathbb{X}: P(\mu, x)\textrm{ holds}\right\}\textrm{ and }\mu_\mathcal{N}^\perp\mres\left\{x\in\mathbb{X}:Q(\mu, x)\textrm{ holds} \right\}.\]
\end{problem} 
That is, we seek to find pointwise properties $P(\mu,x)$ and $Q(\mu, x)$ that identify the part of $\mu$ where the underlying geometric structure agrees with the structure of sets in $\mathcal{N}$ and the part of $\mu$ where the underlying geometric structure is distinct from that of the sets in $\mathcal{N}$.
There is particular interest in understanding the conditions under which measures can be decomposed when $\mathbb{X}$ is a metric space, $\mathcal{M}$ contains the Borel sets, and $\mathcal{N}$ is the collection of \textit{rectifiable curves}, that is, compact, connected sets of finite length. 
Measures $\mu$ which are carried by the collection of rectifiable curves are called \textit{rectifiable measures}, and measures which are singular to the collection of rectifiable curves are called \textit{purely unrectifiable measures}.  We use the notation 
\begin{equation}
\label{eq:decomposition_notation}
\mu=\mu_{\text{rect}}+\mu_\text{pu}
\end{equation} to indicate decomposition of the measure $\mu$ into a rectifiable component and a purely unrectifiable component.
We remark that the class of rectifiable curves agrees with the class of images of the unit interval under Lipschitz maps, $f([0,1])$, where $f:[0,1]\rightarrow\mathbb{X}$ is Lipschitz.   Therefore, in our discussion of rectifiable and purely unrectifiable measures we will freely move between discussing compact, connected sets of finite length and images of Lipschitz maps.

The study of rectifiable measures stems from the study of rectifiable sets.  
A \textit{rectifiable set} is a set which is contained $\calH^1$-a.e. in a countable union of rectifiable curves, where $\calH^1$ denotes the 1-dimensional Hausdorff measure. For an introduction to Hausdorff measures, see e.g. \cite[Section 4.3]{M95}. Given an arbitrary set in $\rr^n$ of finite length, we cannot expect the set to admit tangent lines at typical points. However, by Rademacher's Theorem, a Lipschitz map $f:[0,1]\rightarrow\rr^n$ is differentiable $\mathcal{L}^1$-a.e., and thus at $\calH^1$-a.e. $x\in f([0,1])$ there is a unique tangent given by the derivative map.  A rectifiable set can inherit the tangents from the rectifiable curve in which it is contained. 
The notion of rectifiable sets in the plane was originally introduced by Besicovitch \cite{B28}. Morse and Randolph \cite{MR44} and Federer \cite{Fed47} extended the concept of rectifiable sets to measures in Euclidean space. Since then a large theory has been developed for identifying rectifiable measures (and their higher-dimensional analogues) $\mu$ under the additional assumption of absolute continuity of $\mu$ with respect to 1-dimensional Hausdorff measure ($\mu\ll\calH^1$). 
Imposing the absolute continuity assumption on measures allows one to replace the class of Lipschitz images with the class of bi-Lipschitz images or Lipschitz graphs in the definition of rectifiable measure. For results in this direction see \cite{M75}, \cite{P87}, \cite{AT15}, \cite{ATT18}, \cite{D19A} and \cite{D19B}.   However, Garnett, Killip, and Schul \cite{GKS10} constructed a doubling measure on $\rr^n$ which is both carried by Lipschitz images and singular to every bi-Lipschitz image.  
Thus they showed that the class of Radon measures carried by bi-Lipschitz images is strictly smaller than the class of measures carried by Lipschitz images. In what follows, we adopt Federer's convention \cite{Fed47}, \cite{F69} and do not assume a priori that $\mu$ is absolutely continuous with respect to $\calH^1$.  

Badger and Schul \cite{BS15}, \cite{BS17} characterized rectifiable Radon measures on $\rr^n$ in terms of $L^2$ Jones' beta numbers and a density adapted Jones function.  Jones' beta numbers for sets were originally introduced by Peter Jones \cite{J90} as a means to solve his Analyst's Traveling Salesman Problem that asked to give necessary and sufficient conditions for a set to be contained in a single rectifiable curve.  Jones provided a solution to his problem for sets in $\rr^2$ and Okikiolu \cite{O92} extended the result to $\rr^n$.  Later, Jones' result was extended to Hilbert space by Schul \cite{S07}.  We summarize the result for Hilbert space here.

%

\begin{definition}[Beta number] Let $E\subset H$, where $H$ is a separable, infinite dimensional Hilbert space or $\rr^n$, and let $Q\subset H$ be bounded.  We define $\beta_E(Q)\in [0,1]$ by
	\[\inf_{\ell}\sup_{x\in E\cap Q}\frac{\dist(x, \ell)}{\diam Q},\]
	where $\ell$ ranges over all lines $\ell$ in $H$.  If $E\cap Q=\emptyset$, we set $\beta_E(Q)=0$.	
\end{definition}

The beta numbers measure how well the set $E$ is approximated by a line in the window $Q$.  
In Euclidean space, dyadic cubes are often a practical choice for windows. 
However, in infinite dimensional Hilbert space, cubes are no longer practical because each cube has infinite diameter and infinitely many children.  To prove an Analyst's Traveling Salesman theorem, Schul replaced dyadic cubes with a \textit{multiresolution family of balls $\mathscr{G}^K$} for a set bounded $K\subset H$, defined as follows.
Fix $k_0$ such that $2^{k_0}\ge\diam(K)$.
For each $k\ge k_0$, let $N_k^K\supset N_{k-1}^K$ be a maximal $2^{-k}$-net for $K$.  Set $U_{k,i}:=B(n_i, 
\lambda_1 2^{-k})$ to be the closed ball of radius $\lambda_1 2^{-k}$ centered at $n_i\in N_k$; we specify $\lambda_1>1$ later. We denote the collection of balls arising from the nets $N_k^K$ by $\mathscr{G}_k^K$ and we set \[\mathscr{G}^K=\bigcup_{k=k_0}^\infty\mathscr{G}_k^K.\]
 Unlike dyadic cubes which are intrinsic to Euclidean space, the multiresolution family depends on the set $K$ as well as on the specific choice of the net $N_k^K.$ 
\begin{theorem}[See \cite{S07}, Theorem 1.1 and Theorem 1.5]
	\label{thm:TSP_Hilbert_Space}
	\color{white}space\color{black}
	\begin{enumerate}
		\item ({Necessary Condition}) Let $\Gamma$ be any connected set containing $K$.  Then 
		\[\sum_{U\in\mathscr{G}^K}\beta_\Gamma^2(U)
		\diam(U)\lesssim\calH^1(\Gamma).\]
		The constant behind the symbol $\lesssim$ depends only on the choice of $\lambda_1$.
		\item ({Sufficient Condition})
		There is a constant $\lambda_0$ such that for all $\lambda_1>\lambda_0$ and for any set $K\subset H$ there exists a connected set $\Gamma_0\supset K$ satisfying \[\calH^1(\Gamma_0)\lesssim\diam(K)+\sum_{U\in\mathscr{G}^{K}}\beta_{K}^2(U)\diam(U).\]Here we require $2^{-k_0}\ge\diam(K)$.		
	\end{enumerate}
\end{theorem}

For the study of rectifiable measures, Jones' beta numbers are replaced by an $L^2$ variant which weigh the distances of points from a line according to the mass distribution of $\mu$.  See e.g. \cite{BS15}, \cite{BS17}.

\begin{definition}[$L^2$ beta number] Let $\mu$ be a locally finite Borel measure on $H$, a separable infinite dimensional Hilbert space or $\rr^n$.  Let $E\subset H$ be a bounded subset.  We define $\beta_2(\mu, E)$ by 
	\[\beta_2^2(\mu, E)=\inf_\ell\int_E\left(\frac{\dist(x, \ell)}{\diam E}\right)^2\frac{d\mu(x)}{d\mu(E)}\]
where the infimum is taken over all lines $\ell$ in $H$. In the case that $\mu(E)=0$, we define $\beta_2(\mu, E)=0$.
\end{definition}
The $L^2$ beta number measures the concentration of mass near lines in a particular window $E$.  To prove results about measures on Euclidean space $\rr^n$, Badger and Schul recorded beta numbers on the collection of half-open dyadic cubes of side length at most $1$, $\Delta_1(\rr^n)$, using the density-normalized $L^2$ Jones function $\tilde{J}_2(\mu, \cdot)$ defined by 
\[\tilde{J}_2(\mu,x):=\sum_{Q\in\Delta_1(\rr^n)}\beta_2^2(\mu, 3Q)\frac{\diam Q}{\mu(Q)}\chi_Q(x)\]
for all $x\in\rr^n$. Similar to above, when $\mu(Q)=0$, we interpret $\beta_2^2(\mu, 3Q)\diam Q/\mu(Q)=0$.
Although Badger and Schul proved results for general Radon measures, here we only state their result for pointwise doubling measures, which has lighter notation.
\begin{theorem}[\cite{BS17}, Theorem E]  Let $n\ge 2$.  If $\mu$ is a Radon measure on $\rr^n$ such that at $\mu$-a.e. $x$ 
	\[\limsup_{r\downarrow 0}\mu(B(x, 2r))/\mu(B(x,r))<\infty\]
then the decomposition $\mu=\mu_\text{rect}+\mu_{pu}$ is given by 
		\begin{equation}\mu_{\text{rect}}=\{x\in \rr^n: \tilde{J}_2(\mu, x)<\infty\},\end{equation}
		\begin{equation}
		\mu_{\text{pu}}=\{x\in \rr^n: \tilde{J}_2(\mu, x)=\infty\}.
		\end{equation}
\end{theorem}

\subsection{Preliminaries}

In this paper, we extend the results of Badger and Schul to pointwise doubling measures on a separable infinite dimensional Hilbert space, $H$.  Following \cite{S07}, we replace the dyadic cubes used in the Euclidean case with a multiresolution family of balls. 
However, we construct the multiresolution family with respect to a carrying set of $\mu$.
Fix some such set $X\subset H$ so that $\mu(H\setminus X)=0$.
For example, we may choose $X=\text{spt}(\mu)$, where $\text{spt}(\mu)$ is the largest closed subset of $H$ such that for all $x$ in the subset, every neighborhood of $x$ has positive measure.
Then fix an integer $k_0$.
We denote a maximal $2^{-k}$-net of $X$ by $X_k^\mu$.  We choose the nets $X_k^\mu$ to be nested so that $X_{k+1}^\mu\supset X_{k}^\mu$ for all $k\ge k_0$.
For a net $X_k^\mu$, we define an associated collection of closed balls,
\[\calC^{\mu}_k=\left\{B(x_k^j, \lambda_22^{-k}):x_k^j\in X_{k}^\mu \right\},\]
where $\lambda_2>1$ is some fixed constant.  We will specify conditions on $\lambda_2$ later in the exposition.  Then we set
\[\calC^\mu:=\bigcup_{k=k_0}^\infty\calC^{\mu}_k,\]
and we call $\calC^\mu$ a  \textit{multiresolution family of balls for the measure $\mu$}.  We emphasize that the collection  $\calC^\mu$  is dependent on the measure $\mu$ and more specifically on the choice of nets $X_k^\mu$.
We use the notation 
\[B_k^j:=B(x_k^j, \lambda_22^{-k}),\]
and for a fixed ball $B\in \calC^\mu$, we denote the center by $x_B$.
For $c>0$, we define 
\[cB_{k}^j=B(x_k^j,c\lambda_2 2^{-k}).\]
That is, $c B$ is the dilation of ball $B$ by a factor of $c$.

We define the $L^2$ density adapted Jones function $\hat{J}(\mu,r, \cdot)$ on  $H$ to be
\[\hat{J}_2(\mu,r, x):=\sum_{\substack{B\in\calC^\mu\\\text{Radius}(B)\le r}}\beta_2^2(\mu, 2B)\frac{\diam B}{\mu(B)}\chi_B(x)\]
for all $x\in H$.   
When $r=1$, we abbreviate 
$\hat{J}_2(\mu, x)=\hat{J}_2(\mu,1,x)$.
The following lemma will be useful.
\begin{lemma}[cf. \cite{BS15}, Lemma 2.9]
	\label{lemma:invariant_under_scale} 
	For every locally finite Borel measure $\mu$, the sets
	\[\{x\in H:\hat{J}_2(\mu, r, x)<\infty\}\text{ and }\{x\in H:\hat{J}_2(\mu, r, x)=\infty\}\] are independent of the parameter $r$.	
\end{lemma}

%

Before we state our main result, we provide the following definition.

\begin{definition}[Pointwise doubling measure]
	We say a measure $\mu$ on $H$ is \textit{pointwise doubling} if $\mu$ is finite on bounded sets and  for $\mu$-a.e. $x$, 
	\[\limsup_{r\downarrow 0}\frac{\mu(B(x, 2r))}{\mu(B(x,r))}<\infty.\]
	Furthermore, we say that $\mu$ is a doubling measure if there exists a constant $D$ such that for all $r>0$ and $\mu$-a.e. $x$,  $\mu(B(x,2r))\le D\mu(B(x,r))$.
\end{definition}

\begin{customthm}{A}[Characterization of rectifiable doubling measures]	\label{cor:characterization_for_doubling}
	Let $\mu$ be a pointwise doubling measure on a separable, infinite dimensional Hilbert space $H$. 
	Then $\mu$ is rectifiable if and only if 
	\[\hat{J}_2(\mu, x)<\infty\text{ for }\mu\text{-a.e.~x}\in H.\]
\end{customthm}
We will freely refer to the \textit{necessary condition} and the \textit{sufficient condition} of Theorem \ref{cor:characterization_for_doubling}.
\begin{center}
	\begin{tabular}{c c}
		Necessary condition:& If $\mu$ is rectifiable, then \begin{math}\hat{J}_2(\mu, x)<\infty\text{ for }\mu\text{-a.e.~ x}\in H.\end{math}\\
		Sufficient condition: &If \begin{math}\hat{J}_2(\mu, x)<\infty\text{ for }\mu\text{-a.e.~ x}\end{math}, then $\mu$ is rectifiable.\\
	\end{tabular}
\end{center}

\begin{customthm}{B}[Decomposition theorem for doubling measures]
	\label{thm:Decomposition}
	Let $\mu$ be a pointwise doubling measure on a separable, infinite dimensional Hilbert space $H$.  Then the decomposition $\mu=\mu_{\text{rect}}+\mu_{\text{pu}}$ is given by 
	\begin{equation}\mu_{\text{rect}}=\{x\in H: \hat{J}_2(\mu, x)<\infty\},\end{equation}
	\begin{equation}
	\mu_{\text{pu}}=\{x\in H: \hat{J}_2(\mu, x)=\infty\}.
	\end{equation}
\end{customthm}

One of the challenges of proving the characterization results in infinite dimensional space as opposed to $\rr^n$ arises in the differences between the multiresolution family of balls and dyadic cubes.  In particular, the set of dyadic cubes satisfies convenient counting properties. 
For a given half-open dyadic cube $Q\in\rr^n$ of side length $2^{-k}$, there are $2^n$ dyadic cubes of side length $2^{-(k+1)}$ contained in $Q$.  Additionally, $cQ$ intersects at most $C(c,n)$ other cubes of side length $2^{-k}$ where $C(c,n)$ is a constant which depends only on the dilation constant $c$ and the dimension of the space $n$. 
The pointwise doubling condition assumed on $\mu$ allows us to recover some of the counting properties of dyadic cubes for subcollections of $\calC^\mu$.
We say a subcollection $\calC'\subset\calC^\mu$ satisfies the \textit{finite overlap condition} with respect to $\mu$ if there exist constants $P_{j-k}^\mu=P(\calC', j-k)$, $j\ge k$, such that for any ball $B=B(x, \lambda_22^{-k})\in\calC'$, there exist at most $P_{j-k}^\mu$ balls $B'=B(y, \lambda_22^{-j})\in\calC'$ satisfying $\mu(B\cap B')>0$.  The proof of the following lemma about doubling measures and the finite overlap condition can be found in the \hyperref[section:Appendix]{appendix}.

\begin{lemma}
	\label{lem:doubling_measure_satisfies_finite_overlap}
	Let $\mu$ be a $D$-doubling measure.  Then $\mu$ satisfies the finite overlap condition.
\end{lemma}

The sufficient direction of the proof of Theorem \ref{cor:characterization_for_doubling} relies on the construction of a rectifiable curve using beta numbers to determine how to connect net points in windows.  
A constructive algorithm for such curves in Euclidean space was presented by Jones \cite{J90} in his proof of the Analyst's Traveling Salesman Theorem.  The algorithm was adapted to infinite dimensional Hilbert space by Schul \cite{S07} who removed the dimensional dependence by more carefully estimating the length of the curve in windows with large beta numbers. 
Badger and Schul \cite{BS17} added flexibility to the algorithm in the Euclidean setting by removing an assumption that subsequent generations of net points be nested.  This flexibility is essential to applications in the setting of measures.  
See also \cite{BNV19} by the author, Badger, and Vellis for an explicit construction algorithm of H\"older maps whose images contain net points in Hilbert space. In the following theorem, we have removed the dimension dependence of constants in the algorithm presented in \cite{BS17} by employing an idea from \cite{S07} and \cite{BNV19}.

\begin{customthm}{C}
\label{thm:TSP_for_nets}
	Let $H$ be a separable, infinite dimensional Hilbert space.  Let $C^*>1$, let $x_0\in H$, $0<\delta\le 1/2$, and $r_0>0$. Let $\{V_k\}_{k=0}^\infty$ be a sequence of nonempty, finite subsets of $B(x_0, C^*r_0)$ such that 
	\begin{enumerate}
		\item[(V1)] distinct points $v,v'\in V_k$ are uniformly separated
		\[|v-v'|\ge \delta^kr_0\]
		\item[(V2)] for all $v_k\in V_k$, there exists $v_{k+1}\in V_{k+1}$ such that \[|v_{k+1}-v_k|<C^*\delta^kr_0.\]
		\item[(V3)] for all $v_k\in V_k$ there exists $v_{k-1}\in V_{k-1}$ such that 
		\[|v_{k-1}-v_k|<C^*\delta^kr_0.\]
	\end{enumerate}
	Suppose that for all $k\ge 1$ and for all $v\in V_k$, we are given a straight line $\ell_{k,v}$ in $H$ and a number $\alpha_{k,v}\ge 0$ such that 
	\[\sup_{x\in (V_{k-1}\bigcup V_k)\cap B(v, 66C^*\delta^{k-2} r_0)}\dist(x, \ell_{k,v})\le\alpha_{k,v}\delta^kr_0,\]	
	and 
	\[\sum_{k=1}^\infty \sum_{v\in V_k}\alpha_{k,v}^2\delta^k r_0<\infty.\]
	Then the sets $V_k$ converge in the Hausdorff metric to a compact set $V\subset B(x_0, C^*r_0)$, and there exists a compact connected set such that $\Gamma\subset\overline{B(x_0, C^*r_0)}$ such that $\Gamma\supset V$ and \[\calH^1(\Gamma)\lesssim_{C^*,\delta} r_0+\sum_{k=1}^\infty\sum_{v\in V_k}\alpha_{k,v}^2\delta^kr_0.\]
\end{customthm}

As illustrated by the example in \cite{GKS10}, studying measures which are carried by Lipschitz images is a distinct problem from studying measures which are carried by Lipschitz graphs.
We define Lipschitz graphs in the following way.  
Let $V$ be an $m$-dimensional plane in $H$.  Let $f: V\rightarrow V^\perp$ be a $L$-Lipschitz map.  Then the set $\Gamma=\{(v, f(v)): v\in V)\}$ is an \textit{$L$-Lipschitz graph} in $H$.
Note that the map $F:V\rightarrow H$ defined by $F(V)=\{(v, f(v)):v\in V\}$ is bi-Lipschitz.  
Lipschitz graphs are characterized by having cone points everywhere in the following sense.
Define the \textit{good cone} at $x$ with respect to $V$ and $\alpha$ by 
\[C_\calG(x, V, \alpha):=\left\{y\in H: \dist(y-x, V)\le\alpha|x-y|\right\},\] 
and the bad cone at $x$ with respect to $V$ and $\alpha$ by 
\[C_\calB(x, V, \alpha):=H\setminus C_\calG(x, V, \alpha).\]
For an $L$-Lipschitz graph and for $x\in\Gamma$, $\Gamma\cap C_\calB(x, V, \alpha)=\emptyset$ where $\alpha\ge\sin(\tan^{-1}(L))$.

In \cite{MM88} Mart\'in and Mattila study sets $E\subset \rr^n$ with $0<\calH^s(E)<\infty$ and $0<s<m\le n-1$, where $s$ is allowed to be non-integer valued.
They define the set $E$ to be $(s,m)$-approximately conically regular if for $\calH^s$-a.e. $x\in E$, there exists an $m$-plane $V$ and $\alpha\in(0,1)$ such that 
\begin{equation}\label{eq: Conically_Regular}\lim_{r\downarrow 0}\frac{\calH^s(E\cap C_\calB(x,r, V, \alpha))}{r^s}=0.\end{equation} 
where $C_\calB(x,r,V,\alpha)$ denotes $C_\calB(x, V,\alpha)\cap B(x,r)$.
Furthermore, they prove that an $s$-set which is $(s,m)$-approximately conically regular is carried $\calH^s$-a.e. by $m$-Lipschitz graphs.
Condition (\ref{eq: Conically_Regular}) serves as an inspiration for the following characterization of graph rectifiable measures, that is, measures carried by Lipschitz graphs.

\begin{customthm}{D}\label{thm:Graph_rectifiable}
	Let $\mu$ be a pointwise doubling measure on a separable, finite or infinite dimensional Hilbert space $H$.  For $\mu$-a.e. $x\in H$ there is an $m$-plane $V$ and an $\alpha\in(0,1)$ such that 
	\begin{equation}\label{eq:Measure_Cone}\lim_{r\downarrow 0}\frac{\mu(C_\calB(x,r,V,\alpha))}{\mu(B(x,r))}=0 \end{equation}
	if and only if $\mu$ is carried by Lipschitz graphs.
\end{customthm}

To explicitly see the connection to the condition (\ref{eq: Conically_Regular}) and condition (\ref{eq:Measure_Cone}), we remark that given a set $E\subset$ with $0<\calH^s(E)<\infty$,
\[\limsup_{r\downarrow 0}\frac{\mathcal{H}^s(E\cap B(x,r))}{r^s}<c<\infty\textrm{ for }\mu\text{-a.e.~ }x\in\rr^n\]
It follows that if 
\[\lim_{r\downarrow 0}\frac{\calH^s(E\cap C_\calB(x, V, \alpha))}{\calH^s(E\cap B(x,r)))}=0\text{ then }\lim_{r\downarrow 0}\frac{\calH^s(E\cap C_\calB(x, V, \alpha))}{r^s}=0.\]
For additional results on densities of measures with respect to cones, see \cite{CKRS10}, \cite{KS08}, and \cite{KS11}.  Graph rectifiability also plays a role in the study of harmonic measure.  See e.g. \cite{AAM19}.

\subsection{Outline}

The proofs of the necessary direction and the sufficient direction of Theorem \ref{cor:characterization_for_doubling} are given sections \ref{section:necessary} and \ref{section:sufficient} respectively.
In Section \ref{section:Decomposition} we combine the results from sections \ref{section:necessary} and \ref{section:sufficient} to give a proof of Theorem \ref{thm:Decomposition}.  
In Section \ref{section:example}, we present an example of a pointwise doubling measure with infinite dimensional support that is carried by Lipschitz images but singular to bi-Lipschitz graphs.
In Section \ref{section:drawing_curves_through_nets} we prove Theorem \ref{thm:TSP_for_nets}, and finally, in Section \ref{section:Graph_rectifiable}, we prove Theorem \ref{thm:Graph_rectifiable}.

\subsection{Acknowledgment}
	I would like to thank my advisor Matthew Badger for insight and guidance throughout this project.
	This work was partially supported by NSF grant DMS 1650546.

\section{Necessary condition for rectifiability}
\label{section:necessary}

The goal of this section is to prove the necessary direction of Theorem \ref{cor:characterization_for_doubling}.
 Throughout we let $H$ denote a separable, finite or infinite dimensional Hilbert space.
 We begin with a theorem about finite measures that satisfy the finite overlap property.

\begin{theorem}
	Let $\nu$ be a finite Radon measure on $H$ whose support is contained in the support of $\mu$.  Let $\Gamma$ be a rectifiable curve, and let $E\subset\Gamma$ such that $\nu(B(x,r))\ge dr$ for all all $x\in E$ and for all $0<r\le r_0$.  Additionally, suppose that $\{B\in\calC^\mu: \nu(B\cap E)>0\}$ satisfies the finite overlap property with constants $P_{j-k}=P(\mu, j-k)$ for $j\ge k$.  
	Then 
	\[\sum_{B\in\bigcup_{k=l}^\infty\calC^\mu_k}\beta_2(\nu, 2B)^2\frac{\diam(B)}{\nu(B)}\int_E\chi_B(x)d\nu\lesssim\calH^1(\Gamma)+\nu(H\setminus\Gamma)\]
	where $l$ is the smallest integer such that $2^{-l}\le r_0$. Here the implied constants depend only on $d$ and $P_{j-k}$.
	\label{thm:necessary}
\end{theorem}

 To prove Theorem \ref{thm:necessary}, we will use a measure-theoretic result for weighted sums.  The proof of the following can by found in the  \hyperref[section:Appendix]{appendix}.
 \begin{lemma}
	\label{lemma:WeightedSumsOfBalls}
	Suppose that $E_0\supset\cdots\supset E_k\supset E_{k+1}\cdots$ and $E=\bigcap E_k$.  Additionally suppose $\nu(E_0)<\infty$, $\omega:E_0\rightarrow[0, \infty)$, $\omega=0$ on $E$, $c_k\ge 0$, and $\sum_{k=0}^{j} c_k\sup_{x\in E_j}\omega(x)\le C<\infty.$
	Then
	\[\sum_{k=0}^\infty c_k\int_{E_k}\omega(x)d\nu(x)=\sum_{j=0}^\infty\sum_{k=0}^j c_k\int_{E_j\setminus E_{j+1}}\omega(x)d\nu(x)\le C\mu(E_0\setminus E).\]
\end{lemma}

We now return to the proof of Theorem \ref{thm:necessary}.
\begin{proof}
Let $\Gamma$ be a rectifiable curve as specified above. We partition $\calC^\mu$ into three subsets:
\[\calC^\mu_\emptyset=\{B: \nu(B\cap E)=0\},\]
\[\calC^\mu_\Gamma=\left\{B\in\bigcup_{k=l}^\infty\calC^\mu_k: \nu(B\cap E)>0\text{ and }\epsilon\beta_2(\nu,2B)\le \beta_\Gamma(\lambda_3B)\right\},\]
\[\calC^\mu_\nu=\left\{B\in\bigcup_{k=l}^\infty\calC^\mu_k: \nu(B\cap E)>0\text{ and }\beta_\Gamma(\lambda_3B)<\epsilon\beta_2(\nu, 2B)\right\},\] where restrictions on $\lambda_3>0$ and $\epsilon>0$ will be specified later.
Now 
\begin{align*}
\sum_{B\in\bigcup_{k=l}^\infty\calC^\mu_k}\beta_2(\nu, 2B)^2\frac{\diam B}{\nu(B)}\int_E\chi_B(x)d\nu
&=\sum_{B\in\bigcup_{k=l}^\infty\calC^\mu_l}\beta_2(\nu, 2B)^2\diam B\frac{\nu(E\cap B)}{\nu(B)}\\
&\le\underbrace{\epsilon^{-2}\sum_{B\in\calC^\mu_\Gamma}\beta_\Gamma(\lambda_3B)^2\diam B}_I+\underbrace{\sum_{B\in\calC^\mu_\nu}\beta_2(\nu, 2B)^2\diam B}_{II}.\\
\end{align*}
We estimate the sums $I$ and $II$ separately.  To estimate $I$ we will invoke Theorem \ref{thm:TSP_Hilbert_Space}, the Traveling Salesman Theorem for sets in Hilbert space.  In order to apply the theorem, we first need to translate from balls centered on the carrying set $X$ to balls centers on the rectifiable curve $\Gamma$.
In doing so, we aim to establish the following bound:	\begin{equation}\sum_{B\in\calC^\mu_\Gamma\cap \calC^\mu_k}\beta_\Gamma(\lambda_3B)^2\diam B\le C\sum_{U\in\mathscr{G}^\Gamma_k}\beta_\Gamma(U)^2\diam U\label{eq:comparison_of_betas}\end{equation}
where $k\ge l$, the constant $C$ is independent of $k$, and $\mathscr{G}_k^\Gamma$ is a multiresolution family of balls for the rectifiable curve $\Gamma$.  The dilation factor $\lambda_1$ for balls in $\mathscr{G}_k^\Gamma$ will be specified below.
To show that (\ref{eq:comparison_of_betas}) holds, fix $B\in\calC_k^\mu$ and let $x_B$ denote the center point. 
Let $g\in B\cap\Gamma$ which exists since $\nu(B\cap\Gamma)\ge \nu(B\cap E)>0$, and let $n_B\in N_k$ such that $\dist(g, n_B)\le 2^{-k}$. 
Now for $y\in \lambda_3B$, 
\[\dist(y, n_B)\le \dist(y, x_B)+\dist(x_B, g)+\dist(g, n_B)\le \lambda_3(\lambda_22^{-k})+\lambda_22^{-k}+2^{-k}<3\lambda_3\lambda_22^{-k}.\]
Thus, by requiring $\lambda_1\ge 3\lambda_3\lambda_2$, we have  $\lambda_3B\subset U=B(n_B, \lambda_12^{-k})$.
As a consequence we can control $\beta_\Gamma(\lambda_3B)$ with $\beta_\Gamma(U)$.  In particular, there is a line $\ell_U$ such that 
\[\beta_\Gamma(U)\ge \frac{1}{2}\sup_{y\in U\cap\Gamma}\left(\frac{\dist(y, \ell_U)}{\diam (U)}\right)\\
\ge \frac{1}{2}\sup_{y\in \lambda_3B\cap\Gamma}\left(\frac{\dist(y, \ell_U)}{\diam(\lambda_3B)}\right)\left(\frac{\diam (\lambda_3B)}{\diam (U)}\right)\\
\ge\frac{\lambda_3\lambda_2}{2\lambda_1}\beta_\Gamma(\lambda_3B).\]

Now fix a ball $U\in\mathscr{G}_k^\Gamma$. We claim that there are at most $P\left(\mu, \lceil\log\frac{2\lambda_1}{\lambda_2}\rceil\right)$ balls $B'\in\calC_k^\mu$ contained in $U$.
To see that this is the case, note that if no balls are contained in $U$ then the bound holds trivially.
Otherwise, fix $B\subset U$.  It follows from triangle inequality that $U\subset 2^{\lceil\log(2\lambda_1/\lambda_2)\rceil} B$.  By the finite overlap property there are at most $P\left(\mu, \lceil\log\frac{2\lambda_1}{\lambda_2}\rceil\right)$ balls $B'\in\calC^\mu_k$ such that $B'\subset 2^{\lceil\log(2\lambda_1/\lambda_2)\rceil}B$, and so there are at most $P\left(\mu,\lceil\log\frac{2\lambda_1}{\lambda_2}\rceil\right)$ balls $B'\subset U$.  
This establishes inequality (\ref{eq:comparison_of_betas}) for each $k\ge l$.
Now summing over all generations $k$ and applying Theorem \ref{thm:TSP_Hilbert_Space} (1) we conclude that 
\[\sum_{B\in\calC^{\mu}_\Gamma}\beta_\Gamma(\lambda_3B)\diam(B)\le \frac{P\left(\mu,\lceil\log\frac{2\lambda_1}{\lambda_2}\rceil\right)\lambda_3\lambda_2}{2\lambda_1}\sum_{U\in\mathscr{G}^\Gamma}\beta_\Gamma(U)\diam(U)\lesssim\calH^1(\Gamma)\]
where the symbol $\lesssim$ depends on $\lambda_2$,$\lambda_3$, and $P\left(\mu, \lceil\log\frac{2\lambda_1}{\lambda_2}\rceil\right)$.
This completes the estimate of sum $I$.  

We now begin the estimation of  $II$.
For $B\in\calC^\mu_\nu\cap\calC^\mu_k$, we fix a line $\ell=\ell_B\in H$ satisfying 
\begin{align*}
\sup_{z\in\Gamma\cap \lambda_3B}\dist(z, \ell)&\le 2\beta_\Gamma(\lambda_3B)\diam(\lambda_3B)\\
&<2\epsilon\beta_2(\nu,2B)\diam(\lambda_3B)\\
&=2\lambda_3\epsilon\beta_2(\nu, 2B)\diam(B)
\end{align*}
The first inequality follows from definition of $\beta_\Gamma(\lambda_3B)$; the second inequality follows from definition of $\calC^\mu_\nu$.  
We partition $2B$ into a set of points near the line $\ell$ and a set of points far from the line $\ell$:
\[N(B)=\{x\in 2B:\dist(x, \ell)\le 2\lambda_3\epsilon\beta_2(\nu, 2B)\diam (B)\},\]
\[F(B)=\{x\in 2B:\dist(x, \ell)>2\lambda_3\epsilon\beta_2(\nu, 2B)\diam(B)\}.\]
Using this partition of the ball $2B$, we see that
\begin{align*}
\beta_2(\nu, 2B)^2&\le\int_N\left(\frac{\dist(x, \ell)}{\diam 2B}\right)^2\frac{d\nu(x)}{\nu(2B)}+\int_F\left(\frac{\dist(x, \ell)}{\diam 2B}\right)^2\frac{d\nu(x)}{\nu(2B)}\\
&\le \lambda_3^2\epsilon^2\beta_2(\nu, 2B)^2+\int_F\left(\frac{\dist(x, \ell)}{\diam (2B)}\right)^2\frac{d\nu(x)}{\nu(2B)}\\
&\le 3\lambda_3^2\epsilon^2\beta_2(\nu, 2B)^2+\int_F\left(\frac{\dist(x, \Gamma\cap \lambda_3B)}{\diam (2B)}\right)^2\frac{d\nu(x)}{\nu(2B)},
\end{align*}
where last inequality follows since
\begin{align*}\left(\frac{\dist(x, \ell)}{\diam 2B}\right)^2&\le\left(\frac{\dist(x, \Gamma\cap \lambda_3 B)}{\diam 2B}+\frac{\dist(\Gamma\cap \lambda_3 B,\ell)}{\diam 2B}\right)^2\\
&\le 2\left(\frac{\dist(x, \Gamma\cap \lambda_3B)}{\diam 2B}\right)^2+2\left(\frac{\dist(\Gamma\cap \lambda_3 B, \ell)}{\diam 2B}\right)^2\\
&\le2\left(\frac{\dist(x, \Gamma\cap \lambda_3B)}{\diam 2B}\right)^2+2\lambda_3^2\epsilon^2\beta_2(\nu, 2B)^2.
\end{align*}
The choice of $\ell$ is used to go between the second and third lines.  
Now since $F\subset 2B$,  $\dist(x,\Gamma)\le\diam 2B=4\lambda_22^{-k}$.
Therefore, if we fix $\lambda_3\ge 6\lambda_2$, then 
\[\dist(x, \Gamma\cap \lambda_3B)=\dist(x, \Gamma).\]
To see this explicitly, let $z\in 2B$ and choose $z_\Gamma$ to be a closest point in $\Gamma$ to $z$.
Then 
\[\dist(z_\Gamma,x_B)\le \dist(z_\Gamma, z)+\dist(z, x_B)\le 4\lambda_22^{-k}+\lambda_22^{-k}<6\lambda_22^{-k}.\]
Once $\lambda_3$ is fixed, choose $\epsilon$ small enough to guarantee that $3\lambda_3^2\epsilon^2<\frac{1}{2}$.
Then we have that for fixed $B$
\[\beta_2(\nu, 2B)^2\le 2\int_F\left(\frac{\dist(x, \Gamma)}{\diam 2B}\right)^2\frac{d\mu(x)}{\mu(2B)}.\]
To prove that  $II$ is finite, it suffices to show that
\[\sum_{B\in\calC^\mu_\nu}\int_F\left(\frac{\dist(x,\Gamma)}{\diam 2B}\right)^2\frac{d\mu(x)}{\mu(2B)}\diam(B)<\infty.\] 
This sum is an improvement in that $\Gamma$ is a fixed reference set which is independent of the window $B$. 

Now observe that for $B\in\calC^\mu_\nu$ there exists $y\in E\cap B$.  
For arbitrary $z\in B(y, 2^{-k})$, an application of the triangle inequality yields
\[\dist(z, x_B)\le \dist(z,y)+\dist(y, x_B)\le 2\lambda_22^{-k},\]
which implies that $B(y, 2^{-k})\subset 2B$.
By the lower regularity assumption on points in $E$,
\[\mu(2B)\ge\mu(B(y, 2^{-k}))\ge d2^{-k}=\frac{d}{2\lambda_2}\diam(B),\] and, in particular, \[\frac{2\lambda_2}{d}\ge\frac{\diam(B)}{\mu(2B)}\]
Using this density estimate, we conclude that
\begin{align*}\sum_{B\in\calC^\mu_\nu}\int_F\left(\frac{\dist(x,\Gamma)}{\diam (2B)}\right)^2\frac{d\mu(x)}{\mu(2B)}\diam B&\le \frac{2\lambda_2}{d}\sum_{B\in\calC^\mu_\nu}\int_{2B}\left(\frac{\dist(x,\Gamma)}{\diam (2B)}\right)^2d\mu(x)\\
&\le \frac{2\lambda_2}{d}\sum_{B\in\calC^\mu_\nu}\frac{1}{(\diam (2B))^2}\int_{2B}\dist^2(x, \Gamma)d\mu(x)\\
&\le \frac{2\lambda_2}{d}\sum_{k=l}^\infty\sum_{B\in\calC^\mu_k}\frac{4^k}{16\lambda_2^2}\int_{2B}\dist^2(x,\Gamma)d\mu(x)\\
&\le\frac{2}{d}\sum_{k=l}^\infty\frac{4^{k-2}}{\lambda_2}P_0^\mu\int_{2B}\dist^2(x,\Gamma)d\mu(x). 
\end{align*}
Here the finite overlap factor $P_0^\mu$ accounts for potential overlapping of ball in $\calC^\mu_k$.  
Set $E_k:=\bigcup_{B\in\calC^{\mu}_k}2B$, $c_k=4^{k-2}$, and $\omega(x)=\dist^2(x,\Gamma)$. Then 
\[\sum_{k=l}^j 4^{k-2}\sup_{x\in E_j}\omega(x)\le\sum_{k=l}^\infty 4^{k-2}(2^{-j+1})^2=\sum_{k=l}^j 4^{k+j-1}<\sum_{j=l}^\infty4^{-j-1}<\frac{4^{-l}}{3}.\]
Furthermore, we verify that $E_{k+1}\subset E_k$ for each $k$. Let $z\in E_{k+1}$, and let $B_z=B(x_z, 2\lambda_22^{-(k+1)})$ denote a ball in $\calC^\mu_k$ that contains $z$.
By maximality of $X_k^\mu$, there is $y_z\in X_k^\mu$ such that 
\[\dist(z, y_z)\le \dist(z, x_z)+\dist(x_z,y_z)\le 2\lambda_22^{-(k+1)}+2^{-k}<2\lambda_22^{-k}.\]
In particular, this implies that $z\in B(y_z, 2\lambda_22^{-k})$.  Of course by definition of $\calC^\mu_k$, $B(y_z, 2\lambda_22^{-k})\subset E_{k}$, so we conclude that $E_{k+1}\subset E_k$.
Thus we may employ the following Lemma \ref{lemma:WeightedSumsOfBalls} to conclude that 
\[\sum_{B\in\calC^\mu_2}\int_F\left(\frac{\dist(x, \Gamma)}{\diam 2B}\right)^2\frac{d\mu(x)}{\mu(2B)}\diam (2B)\le C(d,\lambda_2, P_0))\nu(H\setminus\Gamma).\]
Note, in particular that $C$ is independent of $\Gamma$.
Combining our estimates of $I$ and $II$, we get the estimate
	\[\sum_{B\in\bigcup_{k=l}^\infty\calC^\mu_k}\beta_2^2(\nu, 2B)\frac{\diam(B)}{\nu(B)}\int_E\chi_B(x)d\nu\lesssim\calH^1(\Gamma)+\nu(H\setminus\Gamma).\qedhere.\]
In particular, since $\Gamma$ is a rectifiable curve and $\nu$ is a finite measure, we conclude that the sum is finite.
\end{proof}

\begin{corollary}
	Let $\mu$ be a finite, lower Ahlfors $d$-regular   Borel measure on $H$.  Suppose that $\mu$ is $D$-doubling.
	Then 
	\[\int_\Gamma \hat{J}_2(\mu,x) d\mu(x)<\calH^1(\Gamma)+\mu(H\setminus\Gamma)<\infty.\]  
\end{corollary}

\begin{proof}
	Recall by Theorem \ref{lem:doubling_measure_satisfies_finite_overlap} that $\mu$ satisfied the finite overlap property.  Then this result follows immediately from Theorem \ref{thm:necessary} by setting $\nu=\mu$, observing that 
	\[\int_\Gamma\hat{J}_2(\mu, 2^{-l},x)d\mu(x)=\sum_{B\in\bigcup_{k=l}^\infty}\beta_2^2(\mu, 2B)\frac{\diam (B)}{\mu(B)}\int_\Gamma\chi_{2B}(x)d\mu(x),\]
	and applying Lemma \ref{lemma:invariant_under_scale}.
\end{proof}

\begin{definition}[Hausdorff density] Let $B(x,r)\subset H$ denote the closed ball with center $x\in H$  and radius $r>0$ . We define the  \textit{lower (Hausdorff) $m$-density} at $x$ by
	\[\underline{D}^m(\mu, x):=\liminf_{r\downarrow 0}\frac{\mu(B(x,r))}{r^m}.\]
\end{definition}

We will show that points of zero lower density do not see rectifiable curves.  This will allow us to focus on points with positive lower density for the proof of the necessary condition of Theorem \ref{cor:characterization_for_doubling}.

\begin{theorem}
	If $\mu$ is a pointwise doubling measure on $H$ then $\mu\mres\{x\in H: \underline{D}^1(\mu, x)=0\}$ is purely unrectifiable.
	\label{thm:lower_density_not_rectifiable}	
\end{theorem}

We will use the following two lemma.    Here $\mathcal{P}^1$ denotes the $1$-dimensional packing measure; see \cite[Section 5.10]{M95} for a definition.
For completeness, the proofs of these lemmas are included in the \hyperref[section:Appendix]{appendix}.

\begin{lemma}
	Let $E\subset [0,1]$.   If $f:E\rightarrow H$ is Lipschitz then 
	\[P^1(f(E))\le (\text{Lip} f)P^s(E)\]
	and 
	\[\mathcal{P}^1(f(E))\le (\text{Lip} f)\mathcal{P}^1(E).\]
	\label{lemma:lipschitz_does_not_increase_dimension}
\end{lemma}

\begin{lemma}
	\label{lemma:density_bound_implies_measure_bound}
	Let $A\subset H$ be a bounded set, and suppose that there exists $r_0>0$ and $M<\infty$ such that for every $x\in A$ and $0<r\le r_0$
	\[\mu(B(x, 2r))\le M\mu(B(x,r))\text{ and }\underline{D}^1(\mu, x))\le\lambda.\]
	Then \[\mu(A)\le\lambda\mathcal{P}^1(A).\]
\end{lemma}

We are now ready to complete the proof of Theorem \ref{thm:lower_density_not_rectifiable}.
The outline follows similarly to that of \cite[Theorem 2.7]{BS15}.  
However, we include details to make explicit the use of the pointwise doubling property in the Hilbert space setting.

\begin{proof}[Proof of Theorem \ref{thm:lower_density_not_rectifiable}]
	Suppose $\mu$ is as in the statement of the theorem, and suppose additionally that $\mu$ is rectifiable. Set $A=\{x\in X:\underline{D}^1(\mu, x)=0\}$. 
	We will show that $A$ intersects the image of every Lipschitz map on a set of measure zero, and hence $A$ itself has zero measure. By Lemma \ref{lemma:lipschitz_does_not_increase_dimension}, for any $E\subset[0,1]$,  $\mathcal{P}^1(f(E))\le (\text{Lip} f)\mathcal{P}^1(E)<\infty$.
	Now let 
	\[(A\cap f(E))_{j}^D=\left\{x\in A\cap f(E): \mu(B(x, 2r))\le D\mu(B(x,r))\text{ for all }0<r\le 1/j\right\}.\]
	Since $\mu$ is pointwise doubling, 
	\[\bigcup_{D=1}^\infty\bigcup_{j=1}^\infty (A\cap f(E))_j^D=A\cap f(E).\]  
	Fix some $D$ and $j$. Since $E$ is bounded and $f$ is continuous, $(A\cap f(E))_j^D$ is bounded.
	Now fix some $\lambda>0$, and recall that $\underline{D}^1(\mu,x)=0\le\lambda$ for all $x\in A$ and, in particular, for all $x\in (A\cap f(E))_j^M$.
	By Lemma \ref{lemma:density_bound_implies_measure_bound}, we have that  \[\mu\left((A\cap f(E))_j^D)\right)\le\lambda\mathcal{P}^1\left((A\cap f(E))_j^D\right)<\infty.\]  
	Thus, letting $\lambda\rightarrow 0$, $\mu\left((A\cap f(E))_j^D\right)=0$ for every $E\subset [0,1]$ and every Lipschitz function $f$. 
	Hence $\mu((A\cap f(E)))=0$ for every $E\subset [0,1]$ and every Lipschitz function $f$. 
	Since $\mu$ is rectifiable, we conclude that $\mu(A)=0$.  If follows immediately that for a rectifiable measure $\mu$, $\underline{D}^1(\mu, x)>0$ for $\mu$-a.e. $x\in H$, and conversely that $\mu\mres\{x: \underline{D}^1(\mu,x )=0\}$ is purely unrectifiable.
\end{proof}

With Theorem \ref{thm:lower_density_not_rectifiable} established, it remains to prove the following theorem in order to obtain the necessary condition of Theorem \ref{cor:characterization_for_doubling}.

\begin{theorem}\label{thm:necessary_condition}Let $\mu$ be a pointwise doubling measure on a separable, infinite dimensional Hilbert space $H$. 
If $\mu$ is rectifiable then
\[\hat{J}_2(\mu, x)<\infty\text{ for }\mu\text{-a.e.~x}\in H.\]
\end{theorem}

\begin{proof}
	Let $\mu$ a rectifiable pointwise doubling measure on $H$.  Since $\mu$ is rectifiable, choose a countable family $\{\Gamma_i\}_{i=1}^\infty$ of rectifiable curves to which $\mu$ gives full mass, i.e., $\mu\left(H\setminus\bigcup_{i=1}^\infty\Gamma_i\right)=0$.  
	As a consequence of Theorem \ref{thm:lower_density_not_rectifiable}, $\mu$ has  positive lower density $\mu$-a.e.. This, together with the pointwise doubling property, implies that $\mu$ gives full mass to $\bigcup_{D=1}^\infty\bigcup_{m=1}^\infty\bigcup_{n=1}^\infty E_{m,n}^D$ where 
	\[E_{m,n}^D=\left\{x\in H:\mu(B(x,r))\ge 2^{-m}r\text{ and } \mu(B(x,2r))\le D\mu(B(x, r))\text{ for all }r\in (0, 2^{-n}]\right\}.\] 
	Therefore, to establish the necessary direction of Theorem \ref{cor:characterization_for_doubling}, it suffices to show that $\hat{J}_2(\mu, x)<\infty$ at $\mu$-a.e. $x\in\Gamma_i\cap E_{m,n}^D$ for every $i$,$m$,$n$, and $D$.
	To this end, fix $i$, $m$, $n$, and $D$.  Set $\Gamma=\Gamma_i$ and then set $E=\Gamma\cap E_{m,n}^D$.
	Define \[\calC^{\mu}_E:=\{B\in\calC^\mu: \mu(E\cap B)>0\text{ and }\text{radius}(B)\le \lambda_22^{-(n+3)}\},\]
	We'll show that $\calC^{\mu}_E$ satisfies the finite overlap property.
	Let $B\in\calC^{\mu}_E\cap\calC^\mu_{k}$ for some $k\ge n+3$.  Let $\{B_i\}_{i=1}^c$ be the collection of balls in $\calC^\mu_E\cap\calC^\mu_j$, $j\ge k$, that intersect $B$.
	Then 
	\begin{align*}\mu(4B)\ge\mu\left(\bigcup_{i=1}^c 2B_i\right)\ge\mu\left(\bigcup_{i=1}^c B(e_i, \lambda_22^{-(j+1)})\right)&\ge D^{-N_{j-k}}\sum_{i=1}^c\mu\left(B(e_i, 2^{N_{j-k}}\lambda_22^{-(j+1)})\right)\\&\ge D^{-N_{j-k}}c\mu(4B).\end{align*}
	Here $N_{j-k}$ is the maximum number of times the ball $B(e_i, 2^{-(j+1)})$, a ball centered at a point in $E$ and contained in $2B_i$, must be doubled to guarantee that the dilated ball contains $\mu(4B)$.  Note that $N_{k-j}$ is dependent only on the difference between $j$ and $k$.  We conclude that $c\le D^{N_{j-k}}$, so we may take $P_{j-k}=D^{N_{j-k}}$.
	Now define the measure $\nu$ by 
	\[\nu:=\mu\mres\bigcup_{\calC^{\mu}_E} 2B.\]
	 Of course $\Gamma$ has finite length, and we have $E\subset B(x,\text{length}(\Gamma))$ for any $x\in E$. It follows that $\nu$ has bounded support and hence, by our definition of pointwise doubling measures, $\nu$ is finite.  
	Thus
	\begin{align*}
		\int_E\hat{J}_2(\mu, 2^{-(n+3)},x)d\mu(x)&=\sum_{B\in\calC^\mu}\beta_2(\mu, 2B)^2\frac{\diam B}{\mu(B)}\int_E\chi_B(x)d\mu(x)\\
		&=\sum_{B\in\calC^{\mu}_E}\beta_2(\nu, 2B)^2\frac{\diam B}{\nu(B)}\int_E\chi_B(x)d\nu(x)\\
		&\lesssim\calH^1(\Gamma)+\nu(H\setminus\Gamma)<\infty
	\end{align*}
	where the last inequality follows from Theorem \ref{thm:necessary}.  In particular, we conclude that $\hat{J}_2(\mu, \lambda_2 2^{-(k+3)},x)<\infty$ at $\mu$-a.e. $x\in E$. 
	It follows from Lemma that \ref{lemma:invariant_under_scale}
	\[\hat{J}_2(\mu, x)<\infty\text{ for }\mu\text{-a.e. }x\in E.\]
	Letting $i$, $m$, $n$, and $D$ vary over all natural numbers proves the result.
\end{proof}

\section{Sufficient condition for rectifiability}\label{section:sufficient}
In this section we prove the sufficient condition of Theorem \ref{cor:characterization_for_doubling}.  As mentioned in the introduction, the main machinery  for this proof is Theorem \ref{thm:TSP_for_nets} which is proved in sections \ref{section:defining_curves}-\ref{section:Length_estimates}.  In order establish a setting in which we can apply Theorem \ref{thm:TSP_for_nets}, we  begin this section by defining a tree structure on $\calC^\mu$.

We define the tree structure on the collection $\calC^\mu$ to model the natural nesting structure of dyadic cubes in Euclidean space.  The tree structure here is more complex than the structure for the dyadic cubes, where we can track the lineage of cubes from an initial generation, say cubes of side-length $1$.  This is because, for a fixed generation of net points $X_k^\mu$, we cannot in general choose a dilation of balls centered at the net points such that the balls are simultaneously pairwise disjoint and also covering $H$. 
To define the family structure, we rely on the following lemma.

\begin{lemma}[\cite{S07}, Lemma 3.19]  Given $c\le\frac{1}{4\lambda_2}$ and $J\ge 10$, there exist $J$ families of connected sets in $H$ such that (denoting a single family by $\{Q_k^{j}\}_{k=k_0, j=0}^{k=\infty, j=j_n})$:
	\begin{enumerate}
		\item[(i)] For every $x\in X_k^\mu$ there is exists a unique $j$ such that for $B_k^j\in\calC^\mu$, $cB_k^j\subset Q_k^{j}$\,  for some family where $\text{radius}(B_k^j)=\lambda_22^{-k}$.
		\item[(ii)] $2c\lambda_22^{-k}\le\diam Q_k^{j}\le 2(1+4\cdot 2^{-J+1})c\lambda_22^{-k}$
		\item[(iii)] If $j\ne j'$ then $Q_k^{ j}\cap Q_k^{j'}=\emptyset$ as long as $Q_k^{j}$ and $Q_k^{j'}$ belong to the same family.  In this case, $\dist(Q_k^{j}, Q_k^{j'})\ge 2^{-k-1}$,
		\item[(iv)] If $Q_k^{j}\cap Q_l^{j'}\ne\emptyset$ for $l>k$ and $Q_k^{j}$ and $Q_l^{j'}$ belong to the same family then $Q_l^{j'}\subset Q_k^{j}$.
	\end{enumerate}
	\label{lemma:core_properties}
\end{lemma}
We call the set $Q_k^j$ satisfying properties (i)-(iv) the \textit{core} of the ball $B_k^j\in\calC^\mu$.  The core $Q_k^j$ can be defined in the following way. Fix $k$, and choose $B_k^j\in \calC_k^{\mu}$.  If $j\ne j'$ then $\dist(cB_k^j, cB_k^{j'})>2^{-(k+1)}$ by choice of the net $X_k^{\mu}$.
Set 
\begin{align*}
Q_{k,0}^{j}&:=cB_k^j,\\
Q_{k, i+1}^{ j}&:=Q_{k,i}^{j}\cup\bigcup_{cB_{k+(i+1)J}^{j'}\cap Q_{k, i}^{j}\ne\emptyset}cB_{k+(i+1)J}^{j'},\\
Q_{k}^{j}&:=\lim_{i\rightarrow\infty} Q_{k, i}^{j}.
\end{align*}
Then we define the $k^{th}$ family of cores to be 
\[\mathscr{Q}_k:=\left\{Q_{k+iJ}^{j}: i\in \mathbb{N}\right\},\]
and we denote the collection of all cores by
\[\mathscr{Q}:=\bigcup_{k=k_0}^\infty \mathscr{Q}_k.\]
We remark that the construction of these cores depends on the choices of the constant $c$ and $J$.  
For balls that belong to the same family there are intrinsic tree structures given by inclusion.
We use the tree structures on $\mathscr{Q}_k$ to define a tree structure on the balls in $\calC^\mu$. 
In particular,  
\begin{itemize}
	\item for $l=k+J$, if $Q_k^{j}\cap Q_l^{ j'}\ne\emptyset$ then $Q_l^{j'}\subset Q_k^{j}$, and we say that $B_l^{j'}$ is a \textit{child} of $B_k^{j}$ ($B_k^{j}\succ B_l^{j'}$);
	\item for $k=l-J$, if there exists $j$ such that $Q_k^j\cap Q_l^{j'}\ne\emptyset$ then $Q_k^j\supset Q_l^{j'}$, and we say that $B_k^{j}$ is the \textit{parent} of $B_l^{j'}$,  $(B_l^{j'})^\uparrow=B_k^j$;
	\item for $l\ge k+iJ$ and $i\ge 1$, if there is $Q_k^{j}$ such that $Q_k^{ j}\cap Q_l^{j'}\ne\emptyset$ and $Q_k^{j}$ and $Q_l^{j'}$ belong to the same family then we say that $B_l^{j'}$ is a \textit{descendant} of $B_k^{j}$ and $B_k^j$ is an \textit{ancestor} of $B_l^{j'}$.
\end{itemize}
We extend the parent, child, and descendant relationships to net points $x_k^j$ and $x_l^{j'}$ in the obvious way. By property (iii), when a ball or net point has a parent, the parent is unique.
We say a collection $\calT\subset \calC^\mu$ is a \textit{tree} if
\begin{enumerate}
	\item there exists a unique $B_0\in\calT$ such that for every ball $B\in\calT$, $B$ is a descendant $B_0$.  We denote the ball $B_0$ by $\textrm{Top}(\calT)$ and we call $B_0$ the top of tree $\calT$;
	\item for every $B\in\calT\setminus\{B_0\}$,  $B^\uparrow\in\calT$.
\end{enumerate}
A \textit{branch} of $\calT$ is a sequence of balls 
\[B_0\succ B_1\succ B_2\succ\cdot\cdot\cdot\textrm{ such that  each }B_i\in\calT.\]
A branch is finite if there is some $B_t\in\calT$ such that for all $B_i\in\calT\setminus\{B_t\}$, $B_t\not\succ B_i$.  That is, no child of $B_t$ is contained in the tree.
If a branch is not finite then it is infinite.
We define the \textit{leaves} of the tree $\calT$ to be the set
\[\textrm{Leaves}(\calT):=\bigcup\left\{\lim_{i\rightarrow\infty}\overline B_i: B_0\succ B_1\succ B_2\succ...\text{ is an infinite branch of }\calT \right\}.\]
Here the limit is taken to be the intersection of nested sets, $\lim_{i\rightarrow\infty}\cap_{j=0}^\infty B_j$.
Now we specify $\lambda_2>(1-2^{-J})^{-2}$.  This specification allows us to prove the following containment of children inside of parent balls.

\begin{lemma}\label{lemma:containment_of_children}
	If $B_l^{j'}\prec B_k^j$, then $B_l^{j'}\subset B_k^j$.
\end{lemma}

\begin{proof}
	Since $B_l^{j'}\prec B_k^j$, $(4\lambda_2)^{-1}B_l^{j'}\cap Q_{k,1}^{j}\ne\emptyset$.  This implies that
	\[\left|x_{B_l^{j'}}- x_{B_k^j}\right|\le \sum_{i=1}^l 2^{-k-iJ}\le 2^{-k}\left(\frac{1}{1-2^{-J}}\right).\]
	Fix $y\in B_l^{j'}$, and observe that 
	\[\left|y-x_{B_k^j}\right|\le \left|y- x_{B_l^{j'}}\right|+\left|x_{B_l^{j'}}, x_{B_k^j}\right|\le \left(\lambda_2 2^{-J}+(1-2^{-J}\right)^{-1})2^{-k}.\]
	By the choice of $\lambda_2$ we have
	we have $\lambda_2(1-2^{-J})>\frac{1}{1-2^{-J}}$, so $(\lambda_22^{-J}+(1-2^{-J})^{-1})<\lambda_22^{-k}$. It follows that $y\in B_k^j$.
\end{proof}

Now that a tree structure has been defined on $\calC^\mu$, we will show that a rectifiable curve can be drawn through the leaves of a tree.  We begin with a lemma that relates the center of mass of a set to its $L^2$ beta number.  This is an adaptation of \cite[Lemma 6.4]{L03}.

\begin{lemma}\label{lemma:center_of_mass}
	Let $\mu$ be a Radon measure on $H$, let $E$ be a Borel set of positive diameter such that $0<\mu(E)<\infty$, and let 
	\[z_E:=\int_E z\frac{d\mu(x)}{\mu(E)}\in H\]
	denote the center of mass of $E$ with respect to $\mu$.  
	For every straight line $\ell$ in $H$,
	\[\dist(z_E,\ell)\le\beta_2(\mu, E, \ell)\diam E.\]
\end{lemma}

\begin{proof}
	For every affine subspace $\ell$ in $H$, the function $\dist(\cdot, \ell)^2$ is convex.
	Thus, 
	\[\dist(z_E, \ell)^2=\dist\left(\int_E z\frac{d\mu(z)}{\mu(E)}, \ell\right)^2\le\int_E\dist(z,\ell)^2\frac{d\mu(z)}{\mu(E)}=\beta_2(\mu, E, \ell)^2(\diam E)^2\]
	by Jensen's inequality.
\end{proof}

We will also need the following lemma which says that a compact connected set of finite length is a Lipschitz image.
\begin{lemma}[\cite{S07}, Lemma 3.7]
	\label{lemma:contained_in_curve}
	If $\Gamma\subset H$ is a closed, connected set such
	that $\calH^1(\Gamma)<\infty$, then there exists a Lipschitz map $f : [0, 1]\rightarrow H$
	such that $\Gamma = f([0, 1]).$
	Moreover, $f$ can be found such that $|f(s) - f(t)|\le32\calH^1(\Gamma)|s-t|$ for all $0 \le s, t\le 1.$
\end{lemma}
 
 \begin{lemma}[Drawing rectifiable curves through the leaves of uniformly doubling trees, cf. \cite{BS17} Lemma 7.3]
 	\label{lemma:Drawing_Leaves_Doubling}
	 Let $\mu$ be a finite measure on $H$ and let $1\le d_\calT<\infty$.   If $\calT$ is a tree of balls from the multiresolution family $\calC^\mu$
 	such that 
 	\begin{equation}\label{inequality:doubling_on_balls}\mu(B^\uparrow)\le d_\calT\mu(B)\text{ for all }B\in\calT\end{equation}
 	and
 	\[S_2(\mu, \calT):=\sum_{B\in\calT}\beta_2(\mu, 2B)^2\diam B<\infty.\]
 	Additionally, suppose that $\calT$ satisfies the finite overlap property.
 	Then there exists a rectifiable curve $\Gamma$ in $H$ such that $\Gamma\supset\textrm{Leaves}(\calT)$ and 
 	\[\calH^1(\Gamma)\lesssim\diam \textrm{Top}(\calT)+d_\calT^{6+J}S_2(\mu, \calT).\]	
 \end{lemma}  
 \begin{proof} 	
 	 By dilating and  translating as needed, we may assume that $\textrm{Top}(\calT)=B(0, \lambda)=:B_0$. Deleting irrelevant balls from $\calT$, we may also assume that every ball $B\in\calT$ belongs to an infinite branch of $\calT$.  Our goal is apply Theorem \ref{thm:TSP_for_nets}.  
 	Set parameters $C^*=5\cdot 2^{J}$, $r_0=\diam (\textrm{Top}(\calT))=2\lambda_2$, and $\delta=2^{-J}$ where $J$ is as in Lemma \ref{lemma:core_properties}.
 	For each $B\in\calT$, let $z_B=\int_B z\frac{d\mu(z)}{\mu(B)}$ denote the center of mass of $B$, and for $k\ge 0$, set
 	\[Z_k=\left\{z_B: B\in\calT\cap \calC^\mu_k\right\}.\]
 	Choose $V_k$ to be any maximal $\delta^k r_0$-separated subset of $Z_k$, and fix $x_0\in B_0$.
 	Then 
 	\[V_k\subset Z_k\subset\textrm{Top}(\calT)\subset B(x_0, r_0)\subset B(x_0, C^*r_0)\]
	Clearly, $V_k$ satisfies (V1) of Theorem \ref{thm:TSP_for_nets}.
	It remains to verify that (V2) and (V3) hold.
	We begin with (V2). Let $k\ge 0$ and let $v\in V_k$, say $v=z_{B}$ for some $B\in\calT\cap\calC_k^{\mu}$.  
	As a consequence of our assumption that every ball in $\calT$ belongs to an infinite branch of $\calT$, there exists $R\in\calT$ such that $R$ is a child of $B$.  
			By maximality of $V_{k+1}$, there is $v'=z_{P}\in V_{k+1}$ for some $P\in\calT\cap\calC^{\mu}_{k+1}$ such that $|z_{R}-z_{P}|<\delta^{k+1}r_0$.
			It follows that 
			\begin{align*}
			|z_{B}-z_{P}|&\le|z_{B}-z_{R}|+|z_{R}-z_{P}|\\
			&\le|z_{B}-x_B|+|x_B-x_R|+|x_R-z_{R}|+|z_{R}-z_{P}|\\
			&\le \frac{1}{2}\diam (B)+\frac{1}{2}\diam(Q_B)+\frac{1}{2}\diam(Q_R)+\frac{1}{2}\diam (R)+\delta^{k+1}\\
			&\le 4\cdot \delta^{k}r_0\\
			&<C^*\delta^{k}r_0.
			\end{align*} 
			Thus condition (V2) is satisfied.

			Finally, to check condition (V3), let $k\ge 1$ and let $v\in V_k$, say $v=z_{B}$ for some $B\in\calT\cap\calC^\mu_k$.
			Let $R$ denote the parent of $B$ which necessarily belongs to $\calT$.
			By maximality, there exists $v'=z_{P}\in V_{k-1}$ for some $P$ in the same generation as $R$ with $|z_{P}-z_{R}|<\delta^{k-1}r_0$.
			It follows that 
			\begin{align*}
			|z_{B}-z_{P}|&\le|z_{B}-z_{R}|+|z_{R}-z_{P}|\\
			&\le |z_{B}-x_B|+|x_B-x_R|+|x_R-z_{R}|+|z_{R}-z_{P}|\\
			&\le\frac{1}{2}\diam(B)+\frac{1}{2}\diam(Q_B)+\frac{1}{2}\diam(Q_R)+\frac{1}{2}\diam(R)+\delta^{k-1}\\
			&\le 5\cdot 2^{J}r_0\delta^{k}\\
			&=C^*r_0\delta^{k}.
			\end{align*}
			So (V3) is satisfied as well. 	
			
 	Fix $k\ge 0$ and $v\in V_k$, and choose a ball $B_{k, v}\in\calT$ such that $v=z_{B_{k,v}}$.  That is, $B_{k,v}$ is the ball in $\calC_k^\mu$ whose center of mass is $v$.
 	For each $k\ge 1$ and $v\ge V_k$, let $\hat{B}_{k,v}\in\calT$ denote the minimal ancestor of $B_{k,v}$ satisfying
 	\begin{itemize}
 		\item $\hat{B}_{k,v}\supset B_{k,v}$;
		\item for every $v'\in V_k\cap B(v, 66C^*\delta^{k-2}r_0)$ and $j\in\{k-1, k\}$, $\hat{B}_{k,v}\supset B_{j,v'}$ 
	\end{itemize}	
 	Now $66C^*\delta^{k-2}r_0=66\cdot 2^{2J}C^*\delta^{k}r_0< 2^{7+2J}\lambda_2\delta^{k}$, so
 	\[\frac{\diam\hat{B}_{k,v}}{\diam B_{k,v}}\le 2^{12+2J}\]
 	for all $j\in \{k-1, k\}$ and $v'\in V_k\cap B(v, 66C^*\delta^{k-2}r_0)$.
 	Furthermore, from assumption  (\ref{inequality:doubling_on_balls}) we get that,
 	\begin{equation}\label{inequality:doubling_bound}\frac{\mu(2\hat{B}_{k,v})}{\mu(2B_{j,v'})}\le d_\calT^{12+2J}\end{equation}
 	for all $j\in\{k-1, k\}$ and $v'\in V_j\cap B(v, 66C^*\delta^{k-2}r_0)$.
	Next, let $k\ge 1$ and let $v\in V_k$.  
 	Choose $\ell_{k,v}$ to be any straight line in $X$ such that 
 	\begin{equation}\label{inequality:choose_a_line}\beta_2(\mu, 2\hat{B}_{k,v}, \ell_{k,v})\le 2\beta_2(\mu, 2\hat{B}_{k,v}).\end{equation}  Such line exists by definition of $\beta$-number.
 	By combining estimate (\ref{inequality:doubling_bound}) and (\ref{inequality:choose_a_line}) we see that
 	\begin{align*}
 	\beta_2&(\mu, 2B_{j, v'},\ell_{j, v'})\diam {B_{j,v'}}=\left(\int_{2B_{j,v'}}\left(\frac{\dist(x, \ell_{k,v})}{\diam (2B_{j,v'})}\right)^{2}\frac{d\mu(x)}{\mu(2B_{j,v'})}\right)^{\frac{1}{2}}\diam (B_{j,v'})\\
 	&\le\int_{2\hat{B}_{k, v}}\left(\left(\frac{\dist(x, \ell_{k,v})}{\diam \hat{B}_{k,v}}\right)^2\frac{d\mu(x)}{\mu(2B_{j,v'})}\right)^{\frac{1}{2}}\diam B_{k, v}\left(\frac{\diam B_{j, v'}}{\diam B_{k, v}}\right)\left(\frac{\diam\hat{B}_{k,v}}{\diam B_{j, v'}}\right)\left(\frac{\mu(2\hat{B}_{k,v})}{\mu(2B_{j, v'})}\right)^{\frac{1}{2}}\\
 	&\le 2^{12+2J}d_\calT^{6+J}\beta_2(\mu,2\hat{B}_{k,v}, \ell_{k, v})\diam B_{k, v}\\
 	&\le (4d_\calT)^{6+J}\beta_2(\mu,2\hat{B}_{k,v})\diam B_{k,v}\\
 	&=:\alpha_{k,v}\delta^{k}r_0
 	\end{align*}
 	for all $j\in \{k, k-1\}$ and all $v'\in B(v, 66C^*\delta^{k-2}r_0)$. This verifies the remaining hypothesis of Theorem \ref{thm:TSP_for_nets}.  Now since $\calT$ satisfies the finite overlap property, the number of times a ball $B\in\calT$ appears as $\hat{B}_{k,v}$ is bounded, and the bound depends on at most the finite overlap constant $P(\calT, 12+2J)$. We conclude that there exists a compact, connected set $\Gamma\subset H$ such that 
 	\[\calH^1(\Gamma)\lesssim_{C^*}r_0+\sum_{k=1}^\infty\sum_{v\in V_k}\alpha_{k,v}^2\delta^{k}r_0\lesssim\diam\text{Top}(\calT)+d_\calT^{6+J}S_p(\mu, \calT),\]
 	and $\Gamma\supset V=\lim_{k\rightarrow\infty}V_k$.  By Lemma \ref{lemma:contained_in_curve}, $\Gamma$ is a rectifiable curve.
 			It remains to check that $\Gamma\supset\textrm{Leaves}(\calT)$.
 			
 			Let $y\in\textrm{Leaves}(\calT)$, say $y=\lim_{k\rightarrow\infty}y_k$ for a sequence of points $y_k\in\overline B_k$ corresponding to an some infinite branch $B_0\succ B_1\succ B_2\succ...$ of $\calT$.
 			Let $z_k=z_{B_k}$ denote the center of mass of $B_k$ and let $v_k\in V_k$ be any point which minimizes distance to $z_k$.
 			By maximality of the net $V_k$, $|z_k-v_k|\le \delta^{k}r_0.$
 			Furthermore, since both $z_k$ and $y_k$ live in $B_k$, $|z_k-y_k|\le \diam B_k=\delta^{k}r_0$.
 			Combining these estimates, we get
 			\[|v_k-y|\le |v_k-z_k|+|z_k-y_k|+|y_k-y|\le 2\cdot \delta^{k}r_0+|y_k-y|\text{ for all }k\ge 0.\]
 			Thus $y=\lim_{k\rightarrow\infty}v_k\in\lim_{k\rightarrow\infty}V_k\subset\Gamma$.
 			Since $y\in\textrm{Leaves}(\calT)$ arbitrary, we conclude that $\Gamma\supset\textrm{Leaves}(\calT)$.
 \end{proof}

We now prove a lemma which is an adaptation of \cite[Lemma 5.6]{BS17} to the setting of trees on which $\mu$ satisfies a doubling property.
Let $\calT$ be a tree of balls in $\calC^{\mu}$, and define a $\mu$-normalized sum function by 
\[\hat{S}_{\calT, b}(\mu, x):=\sum_{B\in\calT}\frac{b(B)}{\mu(B)}\chi_{B}(x) \text{ for all }x\in H.\]
We interpret $0/0=0$ and $1/0=\infty$.  
The following result hold.

\begin{lemma}[Localization lemma for doubling tree]
\label{lemma:localization}
Let $\calT\subset\calC^\mu$ be a tree, and suppose that there exists a constant $D_{\calT}$ such that $\mu(B)\le D_{\calT}\mu\left( aB\right)$ for every $B\in\calT$ where $a$ is some fixed constant satisfying some $0<a\le c$.  Here $c$ is as in Lemma \ref{lemma:core_properties}.
Let $b:\calT\rightarrow [0,\infty)$.
Then for all $N<\infty$, and $\epsilon>0$, there exists a partition of $\calT$ into a set $\textrm{Good}(\calT, N, \epsilon)$ of good balls and a set $\textrm{Bad}(\calT, N, \epsilon)$ of bad balls with the following properties.
\begin{enumerate}
	\item[(i)] Either $\textrm{Good}(\calT, N, \epsilon)=\emptyset$ or $\textrm{Good}(\calT, N, \epsilon)$ is tree of balls from $\calC^\mu$ with
	\[\textrm{Top}(\textrm{Good}(\calT, N, \epsilon))=\textrm{Top}(\calT).\]
	\item[(ii)] Every child of a bad ball is a bad ball: if $B$ and $R$ belong to $\calT$, $R\in\textrm{Bad}(\calT, N, \epsilon)$ and $B\prec R$, then $B\in\textrm{Bad}(\calT, N, \epsilon)$.
	\item[(iii)] The set $E:=\{x\in\textrm{Top}(\calT): S_{\calT, b}(\mu, x)\le N\}$ and $E':=E\cap\textrm{Leaves}(\textrm{Good}(\calT,N,\epsilon))$ 
	have comparable measures:
	\[\mu(E')\ge (1-\epsilon\mu(\textrm{Top}(\calT)))\mu(E).\]
	\item[(iv)] The sum of the function $b$ over the good cubes is finite
	\[\sum_{B\in\textrm{Good}(\calT, N,\epsilon)}b(B)<\frac{N D_{\calT}}{\epsilon}.\]
\end{enumerate}
\end{lemma}

\begin{proof}
Suppose that $\calT,$ $\mu$, $b$, $N$, $\epsilon$, $E$, and $E'$ are as given in the statement of the lemma.
If $\mu(E)=0$ then we may declare every ball $B\in\calT$ to be a bad ball, and the conclusion of the lemma holds trivially.  Therefore, suppose that $\mu(E)>0$.
Declare a ball $B\in\calT$ to be a bad ball if there exists a ball $B'\in\calT$ such that $B$ is a descendant of $B'$ and $B'$ satisfies
\[\mu(E\cap B')\le\epsilon\mu(E)\mu\left(Q_{B'}\right)\]
where $Q_{B'}$ is the core of the ball $B'$.
We call $B$ a good ball if $B$ is not a bad ball.
Properties (i) and (ii) are immediately satisfied by the definitions of good and bad balls.
To check property (iii) we remark that $E\setminus\textrm{Leaves}(\textrm{Good}(\calT, N,\epsilon))\subset E\cap\bigcup_{B\in\textrm{Bad}(\calT, N, \epsilon)} B.$ Let $\textrm{Bad}_M(\calT, N, \epsilon)\subset\textrm{Bad}(\calT, N, \epsilon)$ denote the set of maximal bad balls. That is, $B\in\textrm{Bad}_M(\calT, N, \epsilon)$ if no ancestor of $B$ is a bad ball. Then 
\begin{align*}
\mu(E\setminus E')&\le \mu\left(E\cap\bigcup_{B\in\textrm{Bad}(\calT, N, \epsilon)} B\right)\\
&\le\mu\left(E\cap\bigcup_{B\in\textrm{Bad}_M(\calT, N,\epsilon)} B\right)\\
&\le\sum_{B\in\textrm{Bad}_M(\calT, N, \epsilon)}\mu(E\cap B)\\
&\le\epsilon\mu(E)\sum_{B\in\textrm{Bad}_M(\calT, N, \epsilon)}\mu\left(Q_B\right)\\
&\le\epsilon\mu(E)\mu\left(\bigcup_{B\in \textrm{Bad}(\calT, N, \epsilon)}Q_B\right)\\
&\le\epsilon\mu(E)\mu(\textrm{Top}(\calT)).
\end{align*}
Note that for the second inequality we use Lemma \ref{lemma:containment_of_children} and for the penultimate inequality we use that the cores of the maximal balls are disjoint by Property (iv) of Lemma \ref{lemma:core_properties}.  
Thus 
\[\mu(E')=\mu(E)-\mu(E\setminus E')\ge(1-\epsilon\mu(\textrm{Top}(\calT)))\mu(E)\]
so property (iii) holds.

Before we begin the proof of (iv),we recall that by definition of $\calT$ and by the construction of cores $Q_B$, \begin{equation}\label{eq:doubling}\mu(B)\le D_{\calT}\mu\left(aB\right)\le D_{\calT}\mu(Q_B).\end{equation}
Finally, since $\hat{S}_{\calT, b}(\mu, x)\le N$ for all $x\in E$, 
\begin{align*}
N\mu(E)&\ge\int_E \hat{S}_{\calT, b}(\mu, x)d\mu(x)\\&\ge\int_E\sum_{B\in\calT}\frac{b(B)}{\mu(B)}\chi_{B}(x)d\mu(x)\\
&\ge \frac{1}{D_{\calT}}\sum_{B\in\calT}b(B)\frac{\mu(E\cap B)}{\mu(Q_B)}\\&\ge \frac{\epsilon}{D_{\calT}}\mu(E)\sum_{B\in\textrm{Good}(\calT, N, \epsilon)}b(B).\end{align*}
The second to last inequality follows by (\ref{eq:doubling}), and the last equality holds because balls in $\textrm{Good}(\calT, N, \epsilon)$ satisfy $\mu(E\cap B)>\epsilon\mu(E)\mu(Q_B)$.  
We conclude that  $\sum_{B\in\textrm{Good}(\calT, N, \epsilon)}b(B)\le N D_{\calT}/\epsilon$.
\end{proof}

\begin{theorem}
	Let $\mu$ be a Radon measure on $H$.  Then the measure 
	\[\mu\mres\left\{x\in H: \limsup_{r\downarrow 0}\frac{\mu(B(x,2r))}{\mu(B(x,r))}<\infty\text{ and }\hat{J}_2(\mu, x)<\infty \right\}\]
	is $1$-rectifiable.
	\label{thm:sufficient}
\end{theorem}

\begin{proof}
	Fix $x\in H$ such that $\limsup_{x\downarrow 0}\mu(B(x,2r))/\mu(B(x,r))<\infty$.  There exists $1\le \omega_x<\infty$ and $r_x>0$ such that $\mu(B(x,2r))\le 2^{\omega_x}\mu(B(x,r))$ for all $0<r<r_x$. 
	Let $a'$ be an integer such that $2^{a'}\ge 1/c$, where $c$ is as in Lemma \ref{lemma:core_properties}.
	Then
	\[\mu(B)\le\mu(B(x,\diam (B))\le 2^{(a'+1)\omega_x}\mu\left(B\left(x,\text{radius}\left( 2^{-(a'+1)} B\right)\right)\right)\le D_{\lambda_2, \omega_x}\mu\left(cB\right)\]
	for every $B\in\calC^\mu$ such that $x\in 2^{-(a'+1)} B$ and $\text{radius}(B)<r_x$.
	A similar series of inequalities shows that if $B'\prec B$ then $B\subset 2^{J+2}B'$ and 
	\[\mu(B)\le \mu(2^{J+2}B')\le d_{\omega_x}\mu(B')\]
	for some constant $d_{\omega_x}$ depending on $\omega_x$.
	Thus, $x$ belongs to the leaves of the tree
	\[\calT_x=\left\{B\in\calC^\mu: B\prec B_x,  \mu(R)\le D_{\lambda_2, \omega_x}\mu\left(cR\right),  \mu(R^\uparrow)\le d_{\omega_x}\mu(R)\text{ for all }R\in\calC^\mu\text{ s.t. }B\prec R\prec B_x  \right\},\]
	where $B_x\ni x$ is defined to be a maximal ball in a family satisfying $\text{radius}(8B)<r_x$.  
	By Lemma \ref{lemma:containment_of_children}, $x\in\text{Top}(\calT_x)\cap\text{Leaves}(\calT_x)$.	
	Now since each tree $\calT_x$ is determine by $B_x\in\calC^\mu$ and $\calC^\mu$ is countable, we can enumerate the trees 
	\[\left\{\calT_x:\limsup_{x\downarrow 0}\mu(B(x,2r))/\mu(B(x,r))<\infty \right\}=\{\calT_{x_i}, i=1,2,3,...\}\]
	for $x_i\in\text{spt}\,\mu$.
	Thus,
	\begin{align*}&\left\{x\in H: \limsup_{x\downarrow 0}\mu(B(x,2r))/\mu(B(x,r))<\infty\text{ and } \hat{J}_2(\mu,x)<\infty\right\}\\&\hspace{3.5in} \subset\bigcup_{i=1}^\infty\bigcup_{j=1}^\infty\{x\in\text{Top}(\calT_{x_i}):\hat{J}_2(\mu,x)\le j\},\end{align*}
	so it suffice to prove that the measure $\mu\mres A_{y,N}$ is $1$-rectifiable for arbitrary $y$ in the carry set $X$ such that $\hat{J}_2(\mu, y)\le N$ where \[A_{y,N}:=\{x\in\text{Top}(\calT_y):\hat{J}_2(\mu,x)\le N\}.\]
	Fix such $y$ and $N$. 
	Set $\eta_y:=\mu(\text{Top}(\calT_y))$.  
	Given $0<\epsilon<\eta_y$, let $\calT_{y,N,\epsilon}:=\text{Good}(\calT_y, N, \epsilon)\subset\calT_y$ denote the tree given by Lemma \ref{lemma:localization} with $\calT=\calT_y$, $b(Q)=\beta_2(\mu, 2B)^2\diam B$, and $a=c$.  
	Then by Lemma \ref{lemma:localization}, $S_2(\mu, \calT_{y,N,\epsilon})<ND_{\calT_{y, N,\epsilon}}/\epsilon)$ and \[\mu(A_{y,N}\cap\text{Leaves}(\calT_{y,N,\epsilon}))\ge(1-\epsilon\eta_y)\mu(A_{y,N}).\]
	By Lemma \ref{lemma:Drawing_Leaves_Doubling}, there exists a rectifiable curve $\Gamma_{y,N,\epsilon}$ in $H$ such that $\Gamma_{y,N,\epsilon}$ captures a significant portion of the mass of $A_{y,N}$:
	\[\mu(A_{y,N}\setminus\Gamma_{y,N,\epsilon})\le\mu(A_{y,N})-\mu(A_{y,N}\cap\Gamma_{y,N,\epsilon})\le\mu(A_{y,N})-(1-\epsilon\eta_y)\mu(A_{y,N})=\epsilon\eta_y\mu(A_{y,N}).\]
	Finally, for $k\ge 1$, choose $0<\epsilon_k<\eta_y$ such that $\lim_{k\rightarrow\infty}\epsilon_k=0$.  Then
	\[\mu\left(A_{y,N}\setminus\bigcup_{k=1}^\infty\Gamma_{y,N,\epsilon_k}\right)\le\inf_{k\ge 1}\mu(A_{y,N}\setminus\Gamma_{y,N,\epsilon})\le \eta_y\mu(A_{y,N})\inf_{k\ge 1}\epsilon_k=0.\]
	We conclude that $\mu\mres A_{y, N}$ is $1$-rectifiable.  This completes the proof.
\end{proof}
 
An immediate corollary of this result is the sufficient direction of Theorem \ref{cor:characterization_for_doubling}.

\section{Proof of Theorem \ref{thm:Decomposition}}
\label{section:Decomposition}
We are now ready to prove the decomposition result, Theorem \ref{thm:Decomposition}.  
\begin{proof}[Proof of Theorem \ref{thm:Decomposition}]
 Let $\mu$ be a pointwise doubling measure on an infinite dimensional Hilbert space $H$, and partition $H$ into two sets:
 \[R=\{x\in H: \hat{J}(\mu,x)<\infty\}\]
 and 
 \[P=\{x\in H: \hat{J}(\mu,x)=\infty\}.\]
 It is clear that both $R$ and $P$ are Borel sets.
 Since $R$ and $P$ partition $H$, we have 
 \[\mu=\mu\mres R+\mu\mres P\text{ and }\mu\mres R\, \bot\, \mu\mres P.\]
 The decomposition $\mu=\mu_\text{rect}+\mu_\text{pu}$ is unique (see \cite[Theorem 1.2]{BS17}), so to prove Theorem \ref{thm:Decomposition} it suffices to show that $\mu\mres R$ is rectifiable and $\mu\mres P$ is purely unrectifiable.
 By Theorem \ref{thm:sufficient}, $\mu\mres R$ is $1$-rectifiable.
 Additionally,  
 \[\mu\mres P\le\mu\mres\{x\in H: \underline{D}^1(\mu,x)=0\}+\mu\mres\{x\in H:\underline{D}^1(\mu,x)>0\text{ and }\hat{J}_2(\mu, x)=\infty\}.\]
By Theorem \ref{thm:lower_density_not_rectifiable}  $\mu\mres \{x\in H: \underline{D}(\mu,x)=0\}$ is purely unrectifiable, and by Theorem \ref{thm:necessary_condition} $\mu\mres\{x\in H:\underline{D}^1(\mu,x)>0\text{ and }\hat{J}_2(\mu, x)=\infty\}$ is purely unrectifiable.  
Therefore, $\mu\mres P$ is also purely unrectifiable.  This completes the proof of Theorem \ref{thm:Decomposition}.  
 
\end{proof}

\section{An example of pointwise doubling measure with infinite dimensional support}
\label{section:example}
In this section we construct a pointwise doubling measure $\mu$ which has infinite dimensional support, is carried by Lipschitz images, and assigns zero measure to every bi-Lipschitz image.
To construct the measure, we build off a construction by Garnett, Killip, and Schul of a doubling measure on $\rr^n$ which is carried by Lipschitz images but singular to bi-Lipschitz images.
\begin{theorem}[Garnett, Killip Schul \cite{GKS10}]
	For $n\ge 2$ there exists a $1$-rectifiable doubling measure $\nu^n$ with $\text{spt}\, \nu^n=\rr^n$.
\end{theorem}
Let $\nu=\nu^2$ be as in \cite{GKS10}, and let $C_\nu$ denote the doubling constant.
Let $V_0\in H$ be a two dimensional linear plane.
Fix a basis on $H$ so that for $x\in V_0$, $x=(a_1,a_2,0,0,\dots)$ for some $a_1,a_2\in\mathbb{R}$.
By the separability of $H$ choose a dense collection $\{x_i\}_{i=1}^\infty$ of $V_0^\perp$, the orthogonal complement of $V_0$.
Set $V_i=V_0+x_i$.
We identify each $V_i$, $i=0,1,2,3,...$ with $\rr^2$ using the map $\pi_i: V_i\rightarrow\rr^2$ defined by 
\[\pi_i((a_1, a_2, a_3,...))=(a_1, a_2).\]
Let $\{c_i\}_{i=0}^\infty$ be a summable sequence of positive numbers, i.e., $c_i>0$ for each $i$ and $\sum_{i=1}^\infty c_i<\infty.$
Then set $\mu:=\sum_{i=0}^\infty c_i\nu_i$ where $\nu_i(E):=\nu(\pi_i(E\cap V_i))$.  
In particular, for $y\in V_i$
\[\nu_j(B(y, r)):=\nu(B(\pi_j(y), S^{ij}_r)),\]
where \[S^{ij}_r:=\begin{cases}\sqrt{r^2-\dist^2(V_i, V_j)},& if \dist(V_i, V_j)<r\\ 0&\text{otherwise.}\end{cases}\]
Since $\nu$ is rectifiable, $\mu$ is also rectifiable.  That $\mu$ is finite on bounded sets follows from the summability of the sequence $\{c_i\}$ together with the fact that $\nu$ is finite on bounded sets. 
Furthermore, $\mu$-a.e. $y$ is an element of some point $V_i$.
Now fix some such $y$ and denote by  $V_{i_y}$ the plane that contains this $y$.
Choose $N_y>i_y$ such that $\sum_{N_{y+1}}^\infty c_i\le\frac{c_{i_y}}{2}$.
For $2r<\min_{\{1,...,N_y\}\setminus\{i_{y}\}}\dist(V_{i_y}, V_i)>0$, \begin{align*}
\mu(B(y, 2r))&=\sum_{i=1}^{N_y}c_i\nu_i(B(y, 2r))+\sum_{i=N_{y+1}}^\infty c_i\nu_i(B(y,2r))\\
&= c_{i_y}\nu_{i_y}(B(y,2r))+\sum_{i=N_{y+1}}^\infty c_i\nu(B(\pi_i(y), S^{i_yi}_2r))\\
&\le c_{i_y}\nu_{i_y}(B(y,2r))+\nu_{i_y}(B(y, 2r))\sum_{i=N_{y+1}}^\infty c_i\\
&\le c_{i_y}\nu_{i_y}(B(y,2r))+\frac{c_{i_y}}{2}\nu_{i_y}(B(y, 2r))\\
&\le c_{i_y}\nu_{i_y}(B(y,2r))+\frac{\mu(B(y,2r))}{2}.
\end{align*}
It follows that
\[\mu(B(y,2r))\le 2c_{i_y}\nu_{i_y}(B(y,2r))=2c_{i_y}\nu(B(\pi(y),2r))\le 2c_{i_y}C_\nu\nu(B(\pi(y),r))\le 2C_\nu\mu(B(y, r)).\]
Thus for every $\mu\text{-a.e. }y$, \[\limsup_{r\downarrow 0}\frac{\mu(B(y,2r))}{\mu(B(y, r))}\le 2C_\nu\] so  $\mu$ is a pointwise doubling measure.
By the density of the collection $\{x_i\}_{i=1}^\infty$, and since the coefficients $c_i$ were chosen to be nonzero, $\text{spt}(\mu)=H$.

\section{Drawing curves through nets: an Analyst's Traveling Salesman Algorithm}
\label{section:drawing_curves_through_nets}

In this section we prove Theorem \ref{thm:TSP_for_nets}.  The proof follows the same outline as the proof of Proposition 3.6 in \cite{BS17}.  We provide full details to the portions of the proof that require adaptations to the setting of infinite dimensional Hilbert space, and we refer the reader to appropriate sections in \cite{BS17} for portions that follow identically.  
The required adaptations, which serve to remove dimension dependence, draw on ideas from \cite{S07}
and \cite{BNV19}.  We begin by restating Theorem \ref{thm:TSP_for_nets} for convenience.
\begin{customthm}{C}
	Let $H$ be a separable, infinite dimensional Hilbert space.  Let $C^*>1$, let $x_0\in H$, $0<\delta\le 1/2$, and $r_0>0$. Let $\{V_k\}_{k=0}^\infty$ be a sequence of nonempty, finite subsets of $B(x_0, C^*r_0)$ such that 
	\begin{enumerate}
		\item[(V1)] distinct points $v,v'\in V_k$ are uniformly separated
		\[|v-v'|\ge \delta^kr_0;\]
		\item[(V2)] for all $v_k\in V_k$, there exists $v_{k+1}\in V_{k+1}$ such that \[|v_{k+1}-v_k|<C^*\delta^kr_0;\]
		\item[(V3)] for all $v_k\in V_k$ there exists $v_{k-1}\in V_{k-1}$ such that 
		\[|v_{k-1}-v_k|<C^*\delta^kr_0.\]
	\end{enumerate}
	Suppose that for all $k\ge 1$ and for all $v\in V_k$, we are given a straight line $\ell_{k,v}$ in $H$ and a number $\alpha_{k,v}\ge 0$ such that 
	\[\sup_{x\in (V_{k-1}\bigcup V_k)\cap B(v, 66C^*\delta^{k-2} r_0)}\dist(x, \ell_{k,v})\le\alpha_{k,v}\delta^kr_0,\]	
	and 
	\[\sum_{k=1}^\infty \sum_{v\in V_k}\alpha_{k,v}^2\delta^k r_0<\infty.\]
	Then the sets $V_k$ converge in the Hausdorff metric to a compact set $V\subset B(x_0, C^*r_0)$, and there exists a compact connected set such that $\Gamma\subset\overline{B(x_0, C^*r_0)}$ such that $\Gamma\supset V$ and \[\calH^1(\Gamma)\lesssim_{C^*,\delta} r_0+\sum_{k=1}^\infty\sum_{v\in V_k}\alpha_{k,v}^2\delta^kr_0.\]
\end{customthm}

As in \cite{BS17}, we prove Theorem \ref{thm:TSP_for_nets} in three parts.
In section \ref{section:Description_of_curves} we construct sets $\Gamma_k$ by connecting vertices in $V_k$ with straight line segments.  
In section \ref{section:Connected} we verify that the sets $\Gamma_k$ are connected.  
Finally, in section \ref{section:Length_estimates} we justify the length estimate on the limiting set.  
For ease of notation, we assume that $r_0=1$ throughout our construction of the curves.
We will need the following two lemmas.
\begin{lemma}Let $B\subset H$ be a bounded set and let $V_0, V_1,...$ be a sequence of nonempty finite subsets of $B$.
If the sequence satisfies (V2)  and (V3) for some $C^*>0$ and $r_0>0$ then $V_k$ converges in the Hausdorff metric to a closed set $V\subset\overline{B}$.
\label{lemma:convergence_to_compact set}
\end{lemma}

\begin{lemma}[{\cite[Lemma 8.3]{BS17}}]
	Suppose that $V\subset\rr^n$ is a $1$-separated set with $\# V\ge 2$ and there exist lines $\ell_1$ and $\ell_2$ and a number $0\le \alpha\le 1/16$ such that 
	\[\dist(v, \ell_i)\le \alpha\text{ for all }v\in V\text{ and }i=1,2.\]
	Let $\pi_i$ denote the orthogonal projection onto $\ell_i$.  
	There exist compatible identifications of $\ell_1$ and $\ell_2$ with $\rr$ such that $\pi_1(v')\le\pi_1(v'')$ if and only if $\pi_2(v')\le \pi_2(v'')$ for all $v', v''\in V$.  
	If $v_1$ and $v_2$ are consecutive points in $V$ relative to the ordering of $\pi_1(V)$, then 
	\[\calH^1([u_1, u_2])\le(1+3\alpha^2)\cdot \calH^1([\pi_1(u_1), \pi(u_2)])\text{ for all }[u_1, u_2]\subset [v_1, v_2].\]
	Moreover, 
	\[\calH^1([y_1, y_2])\le(1+12\alpha^2)\cdot\calH^1([\pi_1(y_1), \pi_1(y_2)])\text{ for all } [y_1, y_2]\subset\ell_2.\] 
	\label{lemma:ordering}
\end{lemma}
Lemma \ref{lemma:convergence_to_compact set} is an analogue to \cite[Lemma 8.2]{BS17} in the setting of Hilbert space.  However, we present a different proof technique to overcome to fact the closed, bounded sets are not necessarily compact in Hilbert space. The proof can be found in the \hyperref[section:Appendix]{appendix}. 
Although $H$ may be infinite dimensional, we will apply Lemma \ref{lemma:ordering} to $V_k$ for each $k$.  Since $V_k$ is a finite collection of points we may think of $V_k$ as being embedded in $\rr^{n_k}$ where $n_k$ is at least the cardinality of $V_k$.

We fix a parameter $0<\epsilon<1/32$ so that the conclusions of Lemma  \ref{lemma:ordering} hold for $\alpha=2\epsilon$.  This parameter will be used throughout our definition of $\Gamma_k$. 
For each $k$, we partition $V_k$ into a set a vertices with $\alpha_{k,v}$ less than $\epsilon$ and a set of vertices with $\alpha_{k,v}$ greater than or equal to $\epsilon$.  
That is, we set $V_k=\VF_k\bigcup \VN_k$ where $\VF_k=\{v\in V_k: \alpha_{k,v}< \epsilon\}$ and $\VN_k=\{v\in V_k: \alpha_{k,v}\ge \epsilon\}$.  
Our construction of $\Gamma_k$ near a vertex $v$ will depend on whether $v\in \VF_k$ or $v\in\VN_k$.

\subsection{Description of curves} 
\label{section:defining_curves} 
\label{section:Description_of_curves}
We construct curves $\Gamma_k$ to be the union of finitely many closed sets which take two forms.  
\begin{enumerate}
	\item \textit{edges}  $[v',v'']$: closed line segments between vertices $v',v''\in V_k$.  
	\item \textit{bridges} $B[j,w', w'']$: closed sets that connect vertices $w', w''\in V_j$ for some $k_0\le j\le k$ and pass through vertices of generation $j'$ nearby $w'$ and $w''$ for every $j'>j$.  More explicitly, for $j\ge k_0$ and $v\in V_j$ define an extension $e[j, v]$ in the following way.  Given $v_0=v$, pick a sequence of vertices $v_1, v_2,...,$ inductively so that $v_1$ is a vertex in $V_{j+1}$ that is closest to $v_0$, $v_2$ is a vertex in $V_{j+2}$ that is closest to $v_1$, etc. 
Then define 
$e[j,v]:=\overline{\bigcup_{i=0}^\infty [v_i,v_{i+1}]}$.
Once extensions have been chosen, for each generation $j'\ge j$ we define the bridge $B[j,w', w'']$ by 
\[B[j,w', w'']:=e[j,w']\cup [w',w'']\cup e[j,w''].\]
\end{enumerate}
If an edge $[v', v'']$ is included in $\Gamma_k$, then $|v'-v''|<30C^*\delta^{k-1}$.  We will store edges constructed in generation $k$ in a set denoted by $\E(k)$.   
We will store each bridge in one of two sets: $\BF(k)$ or $\BN(k)$.  We will add bridges to $\BF(k)$ when we are constructing a portion of $\Gamma_k$ nearby a vertex $v$ satisfying $\alpha_{k,v}<\epsilon$, and we will add bridges to $\BN(k)$ when we are constructing $\Gamma_k$ for vertices $v$ with $\alpha_{k,v}\ge\epsilon$.  
We denote the set of all bridges by $\B(k):=\BF(k)\cup \BN(k)$.
Bridges are frozen in that if a bridge $B[k, v', v'']$ appears in $\Gamma_k$ for some $k$ then that $B[k, v', v'']$ also appears in $\Gamma_{k'}$ for all $k'\ge k$.
We will need the following definition of semi-flat vertices to build $\Gamma_k$ near non-flat vertices.
\begin{definition}[Semi-flat vertex] For $k\ge k_0+1$,  we call a vertex $y\in V_{k-1}$ a semi-flat vertex if $\alpha_{k-1,y}\ge \epsilon$ and there exists a vertex $v\in V_{k}$ such that $|y-v|\le 32C^*\delta^{k-1}$ and $\alpha_{k,v}\le\epsilon$.
\end{definition}
Given a semi-flat vertex $y\in V_{k-1}$, we can choose vertex $v\in \VF_k$ such that $|v-y|\le 32C^*\delta^{k-1}$.  Then since $B(y, 33C^*\delta^{k-2})\subset B(v,66C^*\delta^{k-2})$, there exists a natural linear ordering on the points in $V_{k-1}\cap B(y,33C^*\delta^{k-2})$ defined in terms of projection onto $\ell_{k,v}$.
We define a set $S_k$ of edges emanating from semi-flat vertices in $V_k$ in the following way.
Fix a semi-flat vertex $y$, and enumerate the points in $V_{k-1}\cap B(y, 33C^*\delta^{k-2})$ from left to right:
\[y_{-l},...,y_{-1}, y_0=y, y_1,..., y_m.\]
Add edges $[y_i, y_{i+1}]$ to $S_k$ for $0\le i\le m-1$ until $|y-i-y_{i+1}|\ge 30C^*\delta^{k-2}$ or until $y_{i+1}$ does not exist.  We symmetrically add edges to $S_k$ between vertices to the left of $y$.

%

If $\#V_k=1$ for infinitely many $k$ then we can choose $\Gamma_k$ to be a singleton and the theorem holds trivially.  Thus let $k_0\ge 1$ be the smallest index such that $\#V_{k}\ge 2$ for all $k\ge k_0$.  
It suffices then to describe the construction of $\Gamma_k$ for $k\ge k_0$.  We will first describe the construction of $\Gamma_{k_0}$.  The subsequent constructions follow by induction on $k$. 

\subsubsection{Base Case: The construction of $\Gamma_{k_0}$}  
We claim that for any $v\in V_{k_0}$, $V_{k_0}\subset B(v, 2C^*\delta^{k_0})\subset B(v, 66C^*\delta^{k_0-2})$.
To see that this is true, recall that by definition of $k_0$, there is a unique element $w\in V_{k_0-1}$.  Additionally, by (V3), for any $v,v'\in V_{k_0}$, we have $|v-w|\le C^*\delta^{k_0}$ and $|v'-w|\le C^*\delta^{k_0}$.
Hence, 
\[|v-v'|\le |v-w|+|w-v'|\le 2C^*\delta^{k_0}.\]
Now suppose that $\VF_{k_0}\ne\emptyset$, and fix some element $v_0$ in the set.
By Lemma \ref{lemma:ordering} there exists a linear ordering on $V_{k_0}$,
\[v_{-l},...,v_{-1}, v_0, v_1,...v_m\]
according to orthogonal projection onto the line $\ell_{k_0, v_0}$.
We connect $v_i$ to $v_{i+1}$ with an edge $[v_i, v_{i+1}]$ for all $-l\le i\le m$.   We store each edge in $\E(k_0)$.

Suppose instead that $\VF_{k_0}=\emptyset$. If there exists $v_0\in \VN_{k_0}$ which is semi-flat with respect to some $y\in V_{k_0+1}$ then the vertices in $V_{k_0}$ can be ordered according to projection on $\ell_{k_0+1, y}$, and we add edges as in the case when $\VF_{k_0}\ne\emptyset.$
Otherwise, enumerate the vertices in $V_{k_0}$ arbitrarily as $v_0,v_1,....,v_m$ and connect $v_i$ to $v_{i+1}$ with the edge $[v_i, v_{i+1}]$ for $0\le i\le m-1$.  We store each edge in $\E(k_0)$.

In any case, we define $\Gamma_{k_0}$ to be the union of edges in $\E(k_0)$.

%
%

\subsubsection{Inductive Case: The construction of $\Gamma_{k}$ from $\Gamma_{k-1}$}

Suppose $\Gamma_{k_0}$,...,$\Gamma_{k-1}$ have been defined for some $k\ge k_0+1$.
To define the next set $\Gamma_k$ we describe the construction of $\Gamma_{k,v}$, the new part of $\Gamma_k$ nearby $v$ for every $v\in V_k$.
 We will first describe the construction of $\Gamma_{k,v}$ for $v\in \VF_k$, and we will subsequently describe the construction of $\Gamma_{k,v}$ for $v\in \VN_k$. 
  We refer to construction near vertices in $\VF_k$ as ``\textbf{Case F} construction'' and construction near vertices in $\VN_k$ as ``\textbf{Case N} construction.''
As mentioned above, edges added in each stage of construction are include in $\E(k)$, and bridges added during \textbf{Case F} are included in $\BF(k)$ whereas bridges added during \textbf{Case N} are included in $\BN(k)$.

\subsubsection*{\textbf{Case F} Construction}
This step of construction follows identically to the case of vertices $v$ satisfying $\alpha_{k,v}<\epsilon$ in Section 8.2 of \cite{BS17} with $30C^*\delta^k$ in place of $30C^*2^{-k}$ and $66C^*\delta^{k-2}$ in place of $65C^*2^{-k}$.
We include further exposition in order to introduce notation that will be used later in the paper.

Fix $v\in\VF_k$.  Identify $\ell_{k,v}$ with $\mathbb{R}$ (and pick a direction ``left" and ``right").
Let $\pi_{k,v}$ denote orthogonal projection onto $\ell_{k,v}$.  
Since $\alpha_{k,v}\le\epsilon$, by Lemma \ref{lemma:ordering} and (V1), both $V_k\cap B(v, 66C^*\delta^{k-2})$ and $V_{k-1}\cap B(v, 66C^*\delta^{k-2})$ can be arranged linearly along $\ell_{k,v}$.  
Set $v_0=v\in V_k$ and let 
\[v_{-l},...,v_{-1},v_0,v_1,...,v_m\]
denote the vertices in $V_k\cap B(v, 66C^*\delta^{k-2})$ arranged from left to right relative to the order of $\pi_{k,v}(v_i)$ in $\ell_{k,v}$.
We will first describe the construction of the ``right half",  $\Gamma_{k,v}^R$, of $\Gamma_{k,v}$.
Starting with $v_0$ and working right, include each closed line segment $[v_i, v_{i+1}]$ as an edge in $\Gamma_{k,v}^R$ until one of the following holds:
\begin{itemize}
	\item  $|v_{i+1}-v_i|\ge 30C^*\delta^{k-1}$
	\item $v_{i+1}\notin B(v, 30C^*\delta^{k-1})$
	\item $v_{i+1}$ is undefined.
\end{itemize}
Let $t\ge 0$ denote the number of edges that were included in $\Gamma_{k,v}^R$.
We consider three subcases:

\noindent{\textbf{Case F-NT:}
If $t\ge 1$ then the vertex $v$ is \textit{non-terminal} to the right, and we are done describing $\Gamma_{k,v}^R$.

\noindent\textbf{Case F-T:}
If $t=0$ then $v$ is a \textit{terminal vertex}.  We determine the construction of $\Gamma_{k}$ be studying the behavior of $\Gamma_{k-1}$ nearby $v$.
Let $w_v$ be a vertex in $V_{k-1}$ that is closest to $v$.  
Enumerate the vertices in $V_{k-1}\cap B(v, 33C^*\delta^{k-2})$ starting from $w_v$ and moving right
\[w_v=w_{v,0}, w_{v,1},...,w_{v,s}.\]
Let $w_{v,r}$ denote the rightmost vertex in $V_{k-1}\cap B(v, C^*\delta^{k-2})$.  There are two alternatives which determine our subcases.  

\noindent\textbf{Case F-T1:}
If $r=s$ or if $|w_{v,r}-w_{v, r+1}|\ge 30C^*\delta^{k-2}$, set $\Gamma_{k,v}^R=\{v\}$.

\noindent\textbf{Case F-T2:}
If $|w_{v,r}-w_{v,r+1}|<30C^*\delta^{k-2}$ (notice the implied existence of $w_{v,r+1}$) then $v_1$ exists by (V2), so it must be that $|v-v_1|\ge 30C^*\delta^{k-1}$. Set $\Gamma_{k,v}^R=B[k,v,v_1]$.

\noindent This completes the description of $\Gamma_{k,v}^R$.  We define the left half, $\Gamma_{k,v}^L$, of $\Gamma_{k,v}$ symmetrically.
Let $\Gamma_k^\text{Flat}:=\bigcup_{v\in \VF_k}\Gamma_{k,v}.$
If $\VN_k=\emptyset$, set \[\Gamma_k=\Gamma_k^\text{Flat}\cup\bigcup_{j={k_0}}^{k-1}\bigcup_{B[j,w', w'']\subset \Gamma_j}B[j, w',w''];\] the construction at stage $k$ is complete.  Otherwise, we will construct $\Gamma_{k,v}$  for $v\in \VN_k$.  We will use these locally defined sets to define $\Gamma_k^{\text{Non-flat}}$ which will then be appended to $\Gamma_k^\text{Flat}$.  The resulting set will be $\Gamma_k$
\medskip

\subsubsection*{{\textbf{Case N}} Construction}
Fix $v\in \VN_{k}$.  We first define $\Gamma_{k,v}$ in terms of a graph.  Let $\mathcal{E}_{k,v}$ be the set of all edges $[v',v'']$ such that $[v',v'']$ is an edge in $\Gamma_k^\text{Flat}$ or in $S_k$ or $B[k,v',v'']$ is a bridge in $\Gamma_{k}^\text{Flat}$, and either $v'$ or $v''$ is in $B(v, 33C^*\delta^{k-2})$.  
Let $\mathcal{V}_{k,v}$ be the set of vertices in $V_k\cap B(v, 33C^*\delta^{k-2})$ together with any additional endpoints of edges in $\mathcal{E}_{k,v}$.  Let $G_{k,v}$ be the graph with edges set $\mathcal{E}_{k,v}$ and vertex set $\mathcal{V}_{k,v}$.
If $G_{k,v}$ is connected then we let $\Gamma_{k,v}$ be the set with edges $[v',v'']$ or bridges $B[k,v',v'']$ such that $[v',v'']\in\mathcal{E}_{k,v}$.  
Otherwise, label the connected components of $G_{k,v}: G_{k,v}^{(1)},...,G_{k,v}^{(n)}$.  
Each connected component contains at least one non-flat vertex, say $v_i$ for $G_{k,v}^{(i)}$.  
Add edge $[v_i,v_{i+1}]$ to a new edge set, $\mathcal{E}'_{k,v}$, for $1\le i\le n-1$.
Then redefine $G_{k,v}$ to be the graph with edge set $\mathcal{E}_{k,v}\bigcup\mathcal{E}_{k,v}'$ and vertex set $\mathcal{V}_{k,v}$.

We now consider the global graph $G_k'$ with edge set $\mathcal{E}_k'=\bigcup_{v\in\VN_k}\mathcal{E}_{k,v}'$ and vertex set $\mathcal{V}_k=\bigcup_{v\in\VN_k}\mathcal{V}_{k,v}$.  
If $G_k'$ contains cycles, we remove edges from $\mathcal{E}_k'$ one-by-one until no cycles remain.  The resulting graph $G_k'$ is a union of trees such that any two vertices which where originally connected are still connected.
We define $\Gamma_k^\text{Non-flat}$ to be the set with edges $[v',v'']$ or bridges $B[k,v',v'']$ such that $[v',v'']\in\left(\bigcup_{v\in V_k}\mathcal{E}_{k,v}\right)\bigcup\mathcal{E}_k'$ and vertex set $\bigcup_{v\in V_k}\mathcal{V}_{k,v}$.  
When $|v'-v''|<30C^*\delta^{k-1}$, we add the new edge  $[v',v'']$ to $\E(k)$ (this includes all edges from $S_k$) and when $|v'-v''|\ge 30C^*\delta^{k-1}$ we add the new bridge $B[k,v',v'']$ to $\BN(k)$.
Finally, we set \[\Gamma_k=\Gamma_k^\text{Flat}\cup\Gamma_k^\text{Non-flat}\cup\bigcup_{j={k_0}}^{k-1}\bigcup_{B[j,w', w'']\subset \Gamma_j}B[j, w',w''].\]

\subsection{Connectedness} 
\label{section:Connected}
We will now prove that $\Gamma_k$ is connected for each $k\ge k_0$.  Again, we rely heavily on the proof of connectedness in \cite{BS17}. We remark that the use of the exponent $k-1$ rather than $k$ in the bound distinguishing between edges and bridge for the case $\alpha_{k,v}<\epsilon$ follows from the use property (V2) in the proof of connectedness.  
We provide details of the proof to highlight where the smaller exponent is needed.

For $k\ge k_0$, every point $x\in\Gamma_k$ is connected to $V_k$ in $\Gamma_k$ because $x$ belongs to an edge $[v',v'']$ between vertices $v',v''\in V_k$ or to some bridge $B[j, u',u'']$ between vertices $u', u''\in V_j$ for some $k_0\le j\le k$.  
Thus, as in \cite{BS17}, to prove that $\Gamma_k$ is a connected set, it suffices to prove that every pair of vertices in $V_k$ is connected in $\Gamma_k$.
We use a double induction scheme as in \cite[Section 8.3]{BS17} to prove that if for any $k\ge k_0+1$, if $\Gamma_{k-1}$ is connected then $\Gamma_k$ is connected. 
  
Our initial induction is on $k$.
For the base case, generation $k_0$, we consider two cases. First suppose that $\VF_{k_0}\ne\emptyset$ or $\VN_{k_0}$ contains a semi-flat vertex.  Then recall there exists a linear ordering on all points in $V_{k_0}$, $v_{-l},...v_0,...v_m$, and $\Gamma_{k_0}$ is constructed by connected by adding an edge $[v_i, v_{i+1}]$ for $-l\le i\le m-1$.  In particular, for $s>r$, $v_r$ is connected to $v_s$ by the sequence of edges $[v_r, v_{r+1}],..., [v_{s-1}, v_{s}]$.  Suppose instead that $\VF_{k_0}=\emptyset$ and $\VN_{k_0}$ does not contain any semi-flat vertex.  Then $\Gamma_{k_0}$ is defined to be a connected graph on the vertices in $V_{k_0}$ so the result holds trivially.
 
Now suppose that $\Gamma_{k-1}$ is connected for some $k\ge k_0+1$. Note that it follows trivially from construction in both the flat case and the non-flat case that if $v',v''\in V_k$ and $|v'-v''|<30C^*\delta^{k-1}$, then $v'$ and $v''$ are connected in $\Gamma_k$. 
Let $x$ and $y$ be arbitrary vertices in $V_k$ and let $w_x, w_y\in V_{k-1}$ denote vertices that are closest to $x$ and $y$ respectively.  
Since $V_{k-1}$ is connected, $w_x$ and $w_y$ can be joined in $\Gamma_{k-1}$ by a tour of $p+1$ vertices in $V_{k+1}$, say, 
\[w_0=w_x, w_1, w_2,..., w_p=w_y\]
where each pair $w_i, w_{i+1}$ of consecutive vertices is connected in $\Gamma_{k-1}$ by an edge $[w_i,w _{i+1}]$ or by a bridge $B[j,u',u'']$ for some $k_0\le j\le k-1$ and $u', u''\in V_j$ with the property that $w_i\in e[j,u]$ and $w_{i+1}\in e[j,u]$.
Set $v_0=x$.  By (V3) and the choice of $w_0$ to be a closest point to $x$, we have, $|v_0-w_0|=|x-w_x|<C^*\delta^k$.

We are now begin our second induction.
For any $0\le t\le p-1$ there exists a vertex $v_t\in V_k$ such that $|v_t-w_t|<C^*\delta^{k-1}$ by (V2).  Assume that $v_0$ and $v_t$ are connected in $\Gamma_k$.
If $t\le p-2$, choose the vertex $v_{t+1}$ to be any vertex in $V_k$ satisfying $|v_{t+1}-w_{t+1}|<C^*\delta^{k-1}$; such vertex exists by (V2).
Otherwise, if $t=p-1$, set $v_{t+1}=v_p=y$, which of course satisfies $|v_{t+1}-w_{t+1}|=|y-w_y|< C^*\delta^k<C^*\delta^{k-1}$ by (V3) and by choice of $w_y$ as the closest vertex in $V_{k-1}$.
We will show that $v_t$ and $v_{t+1}$ are connected in $\Gamma_k$ in order to conclude that $v_0$ and $v_{t+1}$ are connected in $\Gamma_k$.  We consider two cases:
	\begin{enumerate}
		\item $w_t$ and $w_{t+1}$ are connected by a bridge.
		\item $w_t$ and $w_{t+1}$ are connected by an edge.
	\end{enumerate}	
	
First suppose that $w_t$ and $w_{t+1}$ are connected by a bridge $B[j, u', u'']$ for $u', u''\in V_j$ where $k_0\le j\le k-1$.  In particular, suppose $w_t\in e[j,u']$ and $w_{t+1}\in e[j,u'']$.
Let $z'$ denote the point in $V_k\cap e[j, u']$ and $z''$ denote the point in $V_k\cap E[j,u'']$.
Since $z', z''\in B[j,u',u'']\subset\Gamma_k$ and bridges are connected subsets of $\Gamma_k$, $z'$ and $z''$ are connected in $\Gamma_k$.
Now by definition of extension in terms of nearest points and by (V2), $|z'-w_t|<C^*\delta^{k-1}$.
Thus
\[|v_t-z'|\le |v_t-w_t|+|w_t-z'|<2C^*\delta^{k-1}<30C^*\delta^{k-1}.\]
An analogous estimation show that $|v_{t+1}-z''|<30C^*\delta^{k-1}.$
It follows that $v_t$ is connected to $z'$ and $v_{t+1}$ is connected to $z''$ so $v_{t}$ is connected to $v_{t+1}$ in $\Gamma_k$.

Secondly, suppose that $[w_t,w_{t+1}]$ is an edge in $\Gamma_{k-1}$.  By definition of edge, we know that  $|w_t-w_{t+1}|<30C^*\delta^{k-2}.$
Hence 
\[|v_t-v_{t+1}|\le |v_t-w_t|+|w_t-w_{t+1}|+|w_{t+1}-v_{t+1}|\le 2C^*\delta^{k-1}+30C^{*}\delta^{k-2}<32C^*\delta^{k-2}.\]
To conclude the proof of the connectedness, we consider two cases depending on whether $\alpha_{k,v_t}<\epsilon$ or $\alpha_{k,v_t}\ge\epsilon$.
When $\alpha_{k, v_t}\ge\epsilon$, we are in the \textbf{Case N} construction of $\Gamma_k$.  In this case, we defined $\Gamma_{k,v_t}$ to be a connected graph with vertices in $B(v_t, 33C^*\delta^{k-2})$ so, in particular, $v_t$ is connected to $v_{t+1}$ in $\Gamma_{k, v_t}$.  
The reduction of edges to construct $\Gamma_k^\text{Non-flat}$ did not affect connectedness.
 
On the other hand, when  $\alpha_{k,v_t}\le\epsilon$ the vertices in $V_k\cap B(v_t, 33C^*\delta^{k-2})$ can be arranged linearly according to the relative ordering under orthogonal projection onto $\ell_{k,v}$.  
We label the vertices in $V_k\cap B(v_t, 32C^*\delta^{k-2})$ lying between $v_t$ and $v_{t+1}$ according to that ordering, 
\[z_0=v_t, z_1,...,z_q=v_{t+1}.\]
Since $(1+3\epsilon^2)32<33$, Lemma \ref{lemma:ordering} guarantees that $v_t, v_{t+1}\in B(z_i, 33C^*\delta^{k-2})$ for all $1\le i\le q$.
Suppose that $\alpha_{k, z_i}<\epsilon$ for all $0\le i\le q$.  Since $\Gamma_{k-1}$ contains the edge $[w_t, w_{t+1}]$, the set $\Gamma_{k,z_i}$ contains either a bridge $B[k, z_i, z_{i+1}]$ or and edge $[z_i, z_{i+1}]$ for each $0\le i\le q-1$ depending on whether $z_i$ is terminal or not terminal to $z_{i+1}$. (We emphasize that \textbf{Case F T1} does not occur here since $w_{t+1}$ exists.)
Hence $z_i$ and $z_{i+1}$ are connected for all $0\le i\le q-1$.  By concatenating paths, we see that $v_t=z_0$ and $v_{t+1}=z_q$ are connected in $\Gamma_k$ as well.
Suppose instead that there exists some $i$ such that $\alpha_{k, z_i}\ge\epsilon$.  Then again by the \textbf{Case F} construction of $\Gamma_{k, z_i}$ as a connected graph, $z_0$ is connected to $z_q$, i.e. $v_t$ is connected to $v_{t+1}$.

By induction, $v_0$ and $v_t$ are connected in $\Gamma_k$ for all $1\le t\le p$.  
In particular, we note that $x=v_0$ and $y=v_p$ are connected in $\Gamma$. Since $x$ and $y$ are arbitrary in $V_k$, it follows that $V_k$ is connected in $\Gamma_k$.  
Again by induction, $\Gamma_k$ is connected for all $k>k_0$.

\subsection{Length estimates}
\label{section:Length_estimates}
The goal of this section is to find length estimates for $\Gamma_{k}$, $k\ge k_0$ which then provide the desired bound for the length of the limiting curve $\Gamma$.  
We first bound the length of $\Gamma_{k_0}$ either in terms of $C^*\delta^{k_0}$ or by the sum over $\alpha_{k_0,v}$ over $v\in V_{k_0}$.  We then bound $\calH^1(\Gamma_k)$ by $\calH^1(\Gamma_{k-1})+C\sum_{v\in V_k}\alpha_{k,v}^2\delta^k$ for all $k\ge k_0+1$ where $C$ is independent of $k$.  We follow
the outline of \cite{BS17} and indicate changes required, particularly near vertices $v\in\VN(k)$ and for the \textbf{Case F-NT} construction.  Before we begin the estimates, we introduce the notion of ``phantom length.''

\subsubsection{Phantom length}
As in \cite{BS17}, we will use phantom length to overcome the challenge of terminal vertices where the old curve does not span the new curve. We define phantom length analogously to the definition in \cite[Section 9.1]{BS17}; we provide the following exposition in order to introduce terminology that will be used in later estimates.

To begin we establish notation to that will allow us to refer to specific vertices in the extensions of a bridge.
For each extension $e[k,v]$, say
\[e[k,v]=\overline{\bigcup_{i=1}^\infty[v_i,v_{i+1}]}\]
define the corresponding extension index set $I[k,v]$ by 
\[I[k,v]=\{(k+i, v_i), i\ge 1\}.\]
Then for each bridge, $B[k, v', v'']$, we define the corresponding bridge index set $I[k, v', v'']$ by
\[I[k,v',v'']=I[k, v']\cup I[k,v''].\]
For all generations $k\ge k_0$ and for all vertices $v\in V_k$, we define that phantom length $p_{k,v}:=3C^*\delta^{k-1}$.  In particular, for a $B[k, v', v'']$ between vertices $v', v''\in V_k$ the totality $p_{k,v',v''}$ of phantom length associated to the index set is 
 \[p_{k,v',v''}:=3C^*\sum_{i=0}^\infty \delta^{k+i-1}+3C^*\sum_{j=0}^\infty \delta^{k+j-1}<12C^*\delta^{k-1}\]
We track phantom length in pairs $(k, v)$ so that we can record both the location and  length of the phantom length.
We initialize $\textrm{Phantom}(k_0)$ to be 
\[\textrm{Phantom}(k_0):=\{(k_0, v):v\in V_{k_0}\}.\]

Now suppose that $\textrm{Phantom}(k_0),\dots,\textrm{Phantom}(k-1)$ have been defined for each $k\ge k_0+1$ so that $\textrm{Phantom}(k-1)$ satisfies the following two properties:

\begin{enumerate}
	\item\textbf{Bridge Property:} If a bridge $B[k-1, w', w'']$ is included in $\Gamma_{k-1}$ then $\textrm{Phantom}(k-1)$ contains $I[k-1, w',w'']$.
	\item\textbf{Terminal Vertex Property:} Let $w\in V_{k-1}$ be a terminal vertex, and let $\ell$ be a line such that $\dist(y, \ell)<\epsilon\delta^{k-1}$ for all $y\in V_{k-1}\cap B(w, 30C^*\delta^{k-2})$.    
	Arrange $V_{k-1}\cap B(w, 30C^*\delta^{k-2})$ linearly with respect to the orthogonal projection $\pi_\ell$ onto $\ell$. 
	If there is no vertex to the ``left" of $w$ or to the ``right" of $w$, then $(k-1, w)\in\text{Phantom}(k-1)$.

\end{enumerate}	
	Note that $\text{Phantom}(k_0)$ satisfies the Bridge Property trivially since no bridges are added during the initial stage of construction and satisfies the Terminal Vertex Property trivially since $\text{Phantom}(k_0)$ includes $(k_0, v)$ for every $v\in V_{k_0}$.	 
	We use $\text{Phantom}(k-1)$ as a basis for defining $\text{Phantom}(k)$.	
	In particular, we initialize $\textrm{Phantom}(k)$ by setting it to $\textrm{Phantom}(k-1)$.
	Next, we delete all pairs of the form $(k-1, w)$ or $(k, \tilde{v})$ that appear in $\textrm{Phantom}(k-1)$ from $\textrm{Phantom}(k)$.	
	Finally, for each vertex $v\in V_k$, we include additional pairs in $\textrm{Phantom}(k)$ according to the following rules:
	
	\noindent\textbf{Case F-NT:} If $\alpha_{k,v}<\epsilon$ and $\Gamma_{k,v}^R$ and $\Gamma_{k,v}^L$ are both defined using \textbf{Case F-NT} then $(k,v)$ does not generate any new phantom length.
	
	\noindent\textbf{Case F-T1:} If $\alpha_{k,v}<\epsilon$ and either $\Gamma_{k,v}^R$ or $\Gamma_{k,v}^L$ is defined by \textbf{Case F-T1} then include $(k,v)\in\textrm{Phantom}(k)$.
	
	\noindent\textbf{Case F-T2:}  Suppose $\alpha_{k,v}<\epsilon$ and either $\Gamma_{k,v}^R$ or $\Gamma_{k,v}^L$ is defined using \textbf{Case F-T2}.  When $\Gamma_{k,v}^R$ is defined by \textbf{Case F-T2}, include $I[k, v,v_1]$ as a subset of $\textrm{Phantom}(k)$.  When $\Gamma_{k,v}^L$ is defined by \textbf{Case F-T2}, include $I[k, v_{-1}, v]$ as a subset of $\textrm{Phantom}(k)$.  In particular, in either case $(k,v)$ is included in $\textrm{Phantom}(k)$.

	\noindent\textbf{Case N:} If $\alpha_{k,v}\ge\epsilon$, include $(k, v')$ in $\textrm{Phantom}(k)$ for all vertices $v'\in \VN_k\cap B(v, 33C^*\delta^{k-2})$.  Additionally, include $I[k, v',v'']$ as a subset of $\textrm{Phantom}(k)$ for every bridge $B[k,v',v'']$ in $\Gamma_{k,v}$.
	
	Clearly, $\textrm{Phantom}(k)$ satisfies the bridge property.  
	To check that $\textrm{Phantom}(k)$ satisfies that terminal vertex property, let $v\in V_k$ be a terminal vertex, and suppose that we can find a line $\ell$ such that 
	\[\dist(y, \ell)<\epsilon\delta^k\textrm{ for all }y\in V_k\cap B(v, 30C^*\delta^{k-1}).\]
	Identify $\ell$ with $\rr^n$ and arrange $V_k\cap B(v, 30C^*\delta^{k-1})$ linearly with respect to the orthogonal projection $\pi_\ell$ onto $\ell$.  
	Assume there is no vertex $v'\in V_k\cap B(v, 30C^*\delta^{k-1})$ to the ``left" of $v$ or to the ``right" of $v$ with respect to the ordering under $\pi_\ell$. 	
	If $\alpha_{k,v}\ge \epsilon$, then $(k, \tilde{v})$ was included in $\textrm{Phantom}(k)$ for every $\tilde{v}\in \VN_k\cap B(v, 33C^*\delta^{k-2})$.  
	In particular, $(k,v)$ is in $\textrm{Phantom}(k)$.	
	Otherwise $\alpha_{k,v}<\epsilon$, so $V_k\cap B(v, 30C^*\delta^{k-1})$ is also linearly ordered with respect to orthogonal projection onto $\ell_{k,v}$.
	By Lemma \ref{lemma:ordering}, the orderings agree modulo the choice of orientation for the lines.
	 The assumption that there is no vertex $v'\in V_k\cap B(v, 30C^*\delta^{k-1})$ to the ``left" or to the ``right" translates to the statement that $\Gamma_{k,v}^L$ or $\Gamma_{k,v}^R$ is defined by \textbf{Case F-T1} or \textbf{Case F-T2}, so $(k,v)$ was included in $\textrm{Phantom}(k)$.  
	Therefore, $\textrm{Phantom}(k)$ satisfies the terminal vertex property.

\subsubsection{Cores of Bridges}

For each bridge $B[k,v',v'']\in \BF(k)$ between vertices $v',v''\in V_k$, we define the core $C[k,v',v'']$ of the bridge to be 
\[C[k, v',v'']:=\frac{9}{10}[v',v'']\]
i.e., $C[k,v',v'']$ is the interval of length $\frac{9}{10}$ of the length of $[v',v'']$ that is concentric to $[v',v'']$.  Recall that $\calH^1(B[k,v',v''])\ge 30C^*\delta^{k-1}$ for every bridge $B[k, v',v'']\in\BF(k)$.   Thus the corresponding core also has significant length, 
\[\calH^1(C[l,v',v''])\ge 27C^*\delta^{k-1}.\]
Cores in $\CF(k)$ are disjoint; see \cite[Section 9.2]{BS17}.  We emphasize that here we only define the cores for bridges in $\BF(k)$

\subsubsection{Proof of Theorem \ref{thm:TSP_for_nets}}
\label{section: TSP}


To establish Theorem \ref{thm:TSP_for_nets}, it suffices to prove that
\begin{equation}
	\sum_{[v',v'']\in\E(k_0)}\calH^1([v',v''])+\sum_{(j,u)\in\textrm{Phantom}(k_0)}p_{j,u}\le C\delta^{k_0}+\sum_{v\in V_0}\alpha_{k_0,v}\delta^{k_0},
	\label{inequality:bound_for_base}
\end{equation}
and then that for all $k\ge k_0+1$
\begin{equation}
\begin{split}
&\sum_{[v',v'']\in\textrm{Edges}(k)}\calH^1([v',v''])+\sum_{B[k,v',v'']\in \B(k)}\calH^1(B[k,v',v''])
+\sum_{(j,u)\in\textrm{Phantom}(k)}p_{j,u}\\
&\qquad \le\sum_{[w',w'']\in\textrm{Edges}(k-1)}\calH^1([w',w''])+\sum_{(j,u)\in\textrm{Phantom}(k-1)}p_{j,u}\\&
\qquad \qquad+C\sum_{v\in V_k}\alpha_{k,v}^2\delta^k+\frac{25}{27}\sum_{C[k,v',v'']\in\CF(k)}\calH^1([k,v',v''].),
\end{split}
\label{inequality:bound_for_k}
\end{equation}
where $C$ denotes a constant depending only on $C^*$ and $\delta$.
To see that establishing these bounds is sufficient, iterate (\ref{inequality:bound_for_k}) $k-k_0$ times and then apply (\ref{inequality:bound_for_base}).  See \cite[Section 9.3]{BS17} for details.

%
%
%

\subsubsection{Preliminary Observation}
We begin with a preliminary observation about the lengths of edges and bridges that will be used in the proofs of the two in equalities.
Recall that an edge $[v',v'']$ in the curves $\Gamma_{k_0}, \Gamma_{k_0+1},...$ is included for some $v',v''\in V_k$ only if $|v'-v''|<30C^*\delta^{k-1}$, while a bridge  $B[k,v',v'']\in\B(k)$ is included for some $v',v''\in V_k$ only if $30C^*\delta^{k-1}\le|v-v'|\le 66C^*\delta^{k-2}$.
Furthermore,  the lengths of the extensions are controlled by (V2): For all $k\ge k_0$ and $v\in V_k$, $\calH^1(e[k,v])\le 2C^*\delta^k$.
Thus,  if $B[k,v',v'']\in \B(k)$ then 
\begin{align*}
\calH^1(B[k,v',v''])&\le\calH^1(e[k,v'])+\calH^1([v',v''])+\calH^1(e[k,v''])\\
&\le 4C^*\delta^{k-2}+\calH^1([v',v''])\le \frac{4\delta^2+30}{30}\calH^1([v',v''])<\frac{32}{30}\calH^1([v',v'']).
\end{align*}


\subsubsection{Length Estimates for Base Case $k_0$}
Recall that there are no bridges added during the construction of $\Gamma_{k_0}$
Since $\calH^1([v',v'']))\le 30C^*\delta^{k_0-1}$ for every $[v',v'']\in\E(k_0)$,
\begin{equation}
	\sum_{[v',v'']\in\textrm{Edge}(k_0)}\calH^1([v',v''])\le \#V_{k_0}30C^*\delta^{k_0-1}.
	\label{estimate:edges_and_bridge_intial}
\end{equation} 

Additionally
\begin{equation}
\sum_{(j,u)\in\textrm{Phantom}(k_0)}p_{j, u}= \sum_{v\in V_{k_0}}p_{k_0,v}\le \#V_{k_0}3C^*\delta^{k_0-1}.
\label{estimate:phantom_length_inital}
\end{equation}

Now we consider two cases.  Suppose first that $\VF_{k_0}\ne\emptyset$.  Fix $v_0$ such that  $\alpha_{k_0,v_0}<\epsilon$, and consider the corresponding approximating line $\ell_{k_0, v_0}$.
For any $v_1, v_2\in V_{k_0}$, consider $\pi(v_1), \pi(v_2)$, their respective projections onto $\ell_{k_0,v_0}$.
We have 
\[|\pi(v_1)-\pi(v_2)|\ge |v_1-v_2|-\dist(v_1, \ell_{k_0,v_0})-\dist(v_2, \ell_{k_0,v_0})\ge C^*\delta^{k_0}-2C^*\epsilon\delta^{k_0}>(1-3\epsilon)C^*\delta^{k_0}.\]
Since $\pi(v_i)\in B(v_0, 66C^{*}\delta^{k_0-2})$ we see that $\#V_0\lesssim_{C^*,\delta} 1.$
In particular, \[\sum_{[v',v'']\in\textrm{Edge}(k_0)}\calH^1([v',v''])+\sum_{(j,u)\in\textrm{Phantom}(k_0)}p_{j, u}\lesssim_{C^*,\delta} \delta^{k_0}.\]
Alternatively, suppose that $\VF(k_0)=\emptyset$.  Then for each added line segment in $\Gamma_{k_0}$, the length of the line segment is charged against the large $\alpha_{k_0,v}$ value for a unique $v\in\VN(k)$. We  also charge the phantom length assigned at each vertex $v$ to the large $\alpha_{k_0,v}$ value. That is, 
\[\sum_{[v',v'']\in\textrm{Edge}(k_0)}\calH^1([v',v''])+\sum_{(j,u)\in\textrm{Phantom}(k_0)}p_{j, u}\lesssim_{C^*, \delta} \sum_{v\in V_{k_0}}\alpha_{k_0, v}\delta^{k_0},\]
Combining these two estimates we conclude that inequality (\ref{inequality:bound_for_base}) holds.

\subsubsection{Length Estimates for $k>k_0$}
We are now ready to work on the proof of (\ref{inequality:bound_for_k}).  Note that edges and bridges forming the curve $\Gamma_k$ and ``new" phantom length may appear in the local portion of $\Gamma_k$ near $v$, namely $\Gamma_{k,v}$,  for several vertices $v\in V_k$ but only need to be accounted for once each in order to estimate the left hand side of (\ref{inequality:bound_for_k}).
We will present the length estimates for \textbf{Case N} construction first and then we will present estimates for \textbf{Case F} construction.  We will refer readers to \cite[Section 9.5]{BS17} for some  details of the \textbf{Case F} construction estimates. 
\medskip 

\noindent\textbf{Case N}:
Here we will pay of edges or bridges in $\Gamma_k\setminus\Gamma_k^\text{Flat}$ as well as well as any parts of edges in $B(v, 2C^*\delta^{k-1})$ for $v\in\VN_k$ that were added during a $\textbf{Case F}$ stage of construction.  We will charge the length to the large $\alpha_{k,v}$ value corresponding to  vertices $v\in V_k$.  
By Lemma \ref{lemma:ordering}, for a semi-flat vertex $v\in \VN_k$, the sum of the length of edges in $S_k$ associated to vertex $v$ cannot exceed \[66(1+3\epsilon^2)C^*\delta^{k-1}<67C^*\delta^{k-1}\le \left(\frac{67C^*}{\epsilon}\right)\alpha_{k,v}\delta^{k-1}.\] 
Additionally, since $G_{k}'$ is a union of disjoint trees, each edge $[v,v']$ in $G_k'$ can be assigned uniquely to a vertex, say $v\in \VN_k$. Then since $\alpha_{k,v}\ge\epsilon$, if the corresponding edge $[v,v']$ was added to $\Gamma_k^\text{Non-flat}$ then $\calH^1([v,v'])\le30C^*\epsilon^{-1}\alpha_{k,v'}\delta^{k-1}$. 
If instead the corresponding bridge $B([v,v',k])$ was added in the construction of $\Gamma_{k}^\text{Non-flat}$ then 
\[\calH^1(B[v,v',k])\le\left( \frac{32}{30}\right)66C^*\delta^{k-2},\]
so $\calH^1(B[v,v',k])\le 71C^*\epsilon^{-1}\alpha_{k,v'}\delta^{k-2}$.
Finally, the length of parts of edges in $B(v, 2C^*\delta^{k-1})$ added during a \textbf{Case F} stage of construction is at most $(1+3\epsilon^2)4C^*\delta^k
\le 5C^*\delta^{k-1}$.
Let $\E_\text{Non-flat}(k)$ denote the set of edges in $E(k)$ such that $[v',v'']\in\mathcal{E}_k'$ or $[v',v'']\in S_k$.
Then 
\begin{equation}
\begin{split}
\label{inequality:estimate_non-flat_vertices}&\left(\sum_{[v',v'']\in\E_{\text{Non-flat}}(k)}\calH^1([v',v''])+\sum_{B([v',v'',k])\in \BN(k)}\calH^1(B[v',v'', k])+5C^*\delta^{k-1}\right)\\&
\hspace{4.5in} \lesssim_{C^*, \delta}\sum_{v\in \VN_{k}} \alpha_{k,v}^2\delta^k.
\end{split}
\end{equation}
We emphasize that here we rely on the fact that we constructed $G_k'$ to be the union of trees, so we can charge each edge of $\mathcal{E}'_k$ to a unique vertex $v\in\VN_k$.
We also bound the phantom length as follows
\begin{equation*}
\begin{split}&\sum_{v\in \VN_{k}}\left(p_{k, v}+\sum_{ B[v',v'',k]\in \BN(k,v)}p_{k,v',v''}\right)\\&\qquad \qquad\le \sum_{v'\in \VN_k}\left(3C^*\delta^{k-1}+\sum_{\B[v',v'',k]\in \BN(k,v)}12C^*\delta^{k-1}\right)\lesssim_{C^*, \delta} \sum_{v\in \VN_k }\alpha_{k,v}^2\delta^k.\end{split}\end{equation*}
\medskip

\noindent\textbf{Case F T1}:
This estimate follows identically to as in Section 9.5 of \cite{BS17}.
In particular, 
\[p_{k,v}+\sum_{[v',v'']\in\textrm{Edges}(k)}\calH^1([v',v
]\cap B(v, 2C^*\delta^k))\le p_{k-1, w_{v,r}}.\]
\medskip

\noindent\textbf{Case F T2}: 
Suppose $v$ is terminal to the right with alternative T2.
Recall that in this step we need to a add a bridge in $\BF(k)$.
Write $v_1\in V_k$ and $w_{v,r}, w_{v,r+1}\in V_{k-1}$ for the vertices appearing in the definition of $\Gamma_{k,v}^R$.  In this case, we will pay for $p_{k,v,v_1}$, the length of the bridge $B[k,v,v_1]$ and the length of the edges in $\Gamma_k\cap B(v, 2C^*\delta^{k-1})$.  We will also pay for the length in $\Gamma_k\cap B(v_1, 2C^*\delta^{k-1})$ if we have not already done so.
As previously noted,
\[\calH^1(B[k,v,v_1])\le 4C^*\delta^k+\calH^1([v',v'']),\]
Since $|v-w_{v,r}|<2C^*\delta^k$ and $|v_1-w_{v,r+1}|<2C^*\delta^{k-1}$, it follows that
\[\calH^1(B[k,v,v_1])\le 4C^*\delta^{k-1}+\calH^1([v,v_1])\le 8C^*\delta^{k-1}+\calH^1([w_{v,r}, w_{v,r+1}]).\]
Note that if $w_{v,r}\notin \VF_{k-1}$ then $w_{v,r}$ is a semi-flat vertex.  In either case, the edge $[w_{v,r}, w_{w,r+1}]$ is in $\Gamma_{k-1}$.  
Additionally, the totality of phantom length associated with vertices in $B[k,v,v_1]$ is $12C^*\delta^{k-1}$.  
Unlike in \cite[Section 9.5]{BS17}, we cannot assume $\alpha_{k,v_1}<\epsilon$.  However, if $\alpha_{k,v_1}\ge \epsilon$ then we have already paid for the length of $\Gamma_k\cap B(v_1, 2C^*\delta^k)$. 
In this case,
\begin{align*}\calH^1(B[k,v,v_1])&+p_{k,v,v_1}+\sum_{[v',v'']\in \textrm{Edges}(k)}\calH^1([v',v'']\cap B(v, 2C^*\delta^k)))\\&\le \calH^1([w_{v,r}, w_{v, r+1}])+23C^*\delta^{k-1}\\&\le\calH^1([w_{v,r}, w_{v,r+1}])+\frac{23}{27}\calH^1(C[k,v,v_1])\end{align*}
where $[w_{v,r}, w_{v,r+1}]\in\textrm{Edges}(k-1)$ and $C[k,v,v_1]\in\CF(k)$.
Otherwise, $\alpha_{k,v_1}<\epsilon.$  
In this case, the total length of parts of edges in $\Gamma_{k}\cap B(v, 2C^*\delta^{k-1})\cup B(v_1, 2C^*\delta^{k-1})$ which has not yet been paid for does not exceed $5C^*\delta^{k-1}$ by Lemma \ref{lemma:ordering}.
Altogether these estimates sum to give the bounds
\begin{align*}\calH^1(B[k,v,v_1])&+p_{k,v,v_1}+\sum_{[v',v'']\in \textrm{Edges}(k)}\calH^1([v',v'']\cap B(v, 2C^*\delta^k)\cup B(v_1, 2C^*\delta^k)))\\&\le \calH^1([w_{v,r}, w_{v, r+1}])+25C^*\delta^{k-1}\\&\le\calH^1([w_{v,r}, w_{v,r+1}])+\frac{25}{27}\calH^1(C[k,v,v_1])\end{align*}
where $[w_{v,r}, w_{v,r+1}]\in\textrm{Edges}(k-1)$ and $C[k,v,v_1]\in\CF(k)$.
\medskip

\noindent\textbf{Case F NT}:} Let $[v',v'']$ be an edge between vertices $v',v''\in V_k$ which are not yet wholly paid for.
Then there exists a vertex $v\in V_k$ such that $|v-v'|<30C^*\delta^{k-1}$, $|v-v''|<30C^*\delta^{k-1}$, $|v'-v''|<30C^*\delta^{k-1}$, and $v'$ is immediately to the left (or to the right) of $v''$ relative to the order defined by $\ell_{k,v}$.
Let $[u', u'']$ be the largest closed subinterval of $[v',v'']$ such that $u'$ and $u''$ lie a distance at least $2C^*\delta^{k-1}$ away from \textbf{Case F-T1} and \textbf{Case F-T2} vertices as well as vertices in $\VN_k$.  Note that we already paid for the length within distance $2C^*\delta^{k-1}$ of these three types of vertices.
Applying Lemma \ref{lemma:ordering},
\begin{align*}\calH^1([u',u''])&\le (1+3\alpha_{k,v'}^2)\calH^1([\pi_{k,v'}(u'),\pi_{k,v'}(u'')])\\&=\calH^1([\pi_{k,v'}(u'),\pi_{k,v'}(u'')])+90C^*\alpha_{k,v}^2\delta^{k-1}.
\end{align*}
Without loss of generality, suppose that $u'$ lies to the left of $u''$ relative to the order of their respective projections on $\ell_{k,v'}$.
Let $z'$ denote the first vertex in $V_{k}\cap B(v', 33C^*\delta^{k-2})$ to the left of $u'$, relative to the order of their projection onto $\ell_{k,v}$, such that $\pi_{k,v'}(z')<\pi_{k,v}(u')-C^*\delta^k$.
Analogously, let $z''$ denote the first vertex in $V_k\cap B(v, 33C^*\delta^{k-2})$ to the right of  $u''$, such that 
$\pi_{k,v}(u'')+C^*\delta^k<\pi_{k,v}(z'')$.
The vertex $z'$ as described above always exists since, if $z'\ne v'$ then $|v'-u'|\le C^*\delta^k$.  Thus $v'$ must be a \textbf{Case F-NT} vertex; a similarly conclusion holds for $v''$.
This implies that $|z'-v'|<30C^*\delta^{k-1}$ and $|z''-v''|<30C^*\delta^{k-1}$.
By (V3), we can find $w', w''\in V_{k-1}$ such that $|w'-z'|<C^*\delta^k$ and $|w''-z''|<C^*\delta^{k}.$  
By choice of $w'$ and $w''$,
\[\pi_{k,v'}(w')<\pi_{k,v'}(u')<\pi_{k,v''}(u'')<\pi_{k,v'}(w'').\]
We claim that there exists a sequence of edges in $\Gamma_{k-1}$ connecting $w'$ to $w''$ such that the edges are contained in an $C^*\delta^{k}\epsilon$- neighborhood of $\ell_{k,v}$. 
To see that this claim is true, recall that by (V3) there are $y', y''\in V_{k-1}$ such that $|y'-v'|<C^*\delta^k$ and  $|y''-v''|<C^*\delta^{k-1}$.  If $\alpha_{k-1,y'}<\epsilon$, then there exists an ordering on the points in $V_{k-1}\cap B(y', 66C^*\delta^{k-1})$ given by projection onto $\ell_{k-1, y'}$.  In this case $|w'-y'|\le |w'-z'|+|z'-v'|+|v'-y'|<30C^*\delta^{k-2}$, so a sequence of edges between $w'$ and $y'$ was added during a \textbf{Case F-NT} stage of construction of $\Gamma_{k-1}$. A similar estimation shows that $|y'-y''|<30C^*\delta^{k-2}$ so there is sequence of edges between $y'$ and $y''$.  If instead $\alpha_{k-1, y'}\ge\epsilon$, then $y'$ is a semi-flat vertex, and, by Lemma \ref{lemma:ordering}, the same sequence of edges was added to $\Gamma_{k-1}$ in the \textbf{Case N} construction.  Now $y''$ satisfies $\alpha_{k-1,y''}<\epsilon$ or $y''$ is a semi-flat vertex.  In either case, since $|y''-w''|<30C^*\delta^{k-2}$, there is a sequence of edges connecting $y''$ to $w''$ in $\Gamma_{k-1}$. We emphasize that since $|w'-v|<66C^*\delta^{k-2}$ and $|w''-v|<66C^*\delta^{k-2}$, the edges added during the construction of $\Gamma_{k-1}$ agree with ordering of points according to projection onto $\ell_{k, v}$.  Furthermore, since all $x\in V_{k-1}\cap B(v, 66C^*\delta^{k-2})$ are distance less than $C^*\delta^{k}\epsilon$ away from $\ell_{k,v}$, the portion of $\Gamma_{k-1}$ between $w'$ and $w''$ is distance less than $C^*\delta^{k}\epsilon$ from $\ell_{k,v}$.

We can pay for $\calH^1([\pi_{k,v'}(u'), \pi_{k,v'}(u'')])$ using the portion of edges in the curve $\Gamma_{k-1}\cap B(v, 66C^*\delta^{k-2})$ that lies over the segment $[\pi_{k,v}(u'),\pi_{k,v}(u'')]$.
Thus, 
\[\calH^1([u',u''])\le \calH^1(E_{k-1}(v)\cap\pi^{-1}_{k,v}([\pi_{k,v}(u'), \pi_{k,v}(u'')]))+90C^*\alpha_{k,v}^2\delta^{k-1}\]
where $E_{k-1}(v)$ denotes the union of edges in $\Gamma_{k-1}$ between the vertices in $V_{k-1}\cap B(v, 66C^*\delta^{k-2})$.
It remains to estimate the overlap of the sets of the form 
\[S_{k,v}[u',u'']:=E_{k-1}(v)\cap\pi_{k,v}^{-1}([\pi_{k,v}(u'),\pi_{k,v}(u'')])\]
Since $S_{k,v'}([u', u''])\subset S_{k,v'}([v',v''])$, it suffices to estimate the length of the overlap of sets $S_{k,v'}[v',v'']$.
Suppose that $v_1,v_2,v_3$ are consecutive vertices in $V_k\cap B(v^{(1)}, 66C^*\delta^{k-2})$ such that portions of edges $[v_1, v_2]$ and $[v_2, v_3]$ are being paid for in this $\textbf{Case F-NT}$ stage.  Suppose that that $[v_1, v_2]$ was added during the construction of $\Gamma_{k,v^{(1)}}$ and $[v_2,v_3]$ was added during the construction of $\Gamma_{k,v^{(2)}}$ where $v^{(1)},v^{(2)}\in\VF_k$ are both non-terminal. 
We will show that 
\[\calH^1(S_{k,v^{(1)}}[v_1,v_2]\cap S_{k,v^{(2)}}[v_2,v_3])<40\alpha^2\delta^{k-1}\]
where $\alpha=\max\{\alpha_{k,v^{(1)}},\alpha_{k,v^{(2)}}\}$.
To start, let $\ell_1$ be a line which is parallel to $\ell_{k,v^{(1)}}$ but passes through $v_2$, and similarly let $\ell_2$ be a line which is parallel to $\ell_{k,v^{(2)}}$ and passes through $v_2$.
Let $\pi_i$ denote orthogonal projection onto $\ell_i$ and let $N_i$ denote the closed tubular neighborhood of $\ell_i$ of radius $2\alpha\delta^k$.
Also, let $E_{k-1}(v^{(1)},v^{(2)}):=E_{k-1}(v^{(1)})\cap E_{k-1}(v^{(2)})$.
Then 
\begin{align*} 
S_{k,v^{(1)}}[v_1,v_2]\cap S_{k,v^{(2)}}[v_2,v_3]&\subset E_{k-1}(v^{(1)},v^{(2)})\cap\pi_1^{-1}([\pi_1(v_1),\pi_i(v_2)])\cap N_1\cap\pi_2^{-1}([\pi_2(v_2), \pi_2(v_3)])\cap N_2\\
&=:E_{k-1}(v^{(1)},v^{(2)})\cap S.
\end{align*}
The remainder of the overlap estimate follows identically as in \cite[Section 9.5]{BS17}.
Now we combine all the estimates above to conclude (\ref{inequality:bound_for_k}).

\section{Graph rectifiable measures}
\label{section:Graph_rectifiable}
In this section we will prove Theorem \ref{thm:Graph_rectifiable}.  Throughout $H$ denotes a finite or infinite dimensional Hilbert space. Recall that we define the \textit{good cone} at $x$ with respect to $V$ and $\alpha$ by 
\[C_\calG(x, V, \alpha):=\left\{y\in H: \dist(y-x, V)\le\alpha|x-y|\right\},\] 
and the bad cone at $x$ with respect to $V$ and $\alpha$ by 
\[C_\calB(x, V, \alpha):=H\setminus C_\calG(x, V, \alpha).\]
 We begin by collecting some geometric results that will be useful in the proof of Theorem \ref{thm:Graph_rectifiable}.  The first result can be found in \cite{M95}.  We present the proof, with slight modifications, in the \hyperref[section:Appendix]{appendix} to highlight some important consequences.

\begin{theorem}[Geometric Lemma] Let $F\subset H$, let $V$ be an $m$-dimensional linear plane in $H$, and let $\alpha\in(0,1)$.  If 
	\[F\setminus C_\calG(x,V, \alpha)=\emptyset\text{ for all }x\in F\]
	then $F$ is contained in an $m$-Lipschitz graphs.  In particular, $F\subset\Gamma$ where $\Gamma$ is a Lipschitz graph with respect to $V$ and the Lipschitz constant corresponding to $\Gamma$ is at most $1+1/(1-\alpha^2)^{1/2}$.
	\label{thm:geometric_lemma}
\end{theorem}
Since we are interested in measure-theoretic results up to sets of measure zero we provide a corollary to Theorem \ref{thm:geometric_lemma}.

\begin{corollary}
	\label{corollary:GeometricLemma}
	Let $\mu$ be a Radon measure on $H$, $V$ be an $m$-dimensional linear plane in $H$, $\alpha\in (0,1)$, and $0<r<\infty$.  If for $\mu$-a.e. $x\in H$  
	\begin{equation}\mu(C_\calB(x,r, V, \alpha))=0 \label{eq:empty_cone}\end{equation}
	then $\mu$ is carried by $m$-Lipschitz graphs.
\end{corollary}

\begin{proof}
	
	Let $F$ denote the set of $x\in H$ that satisfy (\ref{eq:empty_cone}). We may assume $F\subset B(0,r/2)$; otherwise we may write $F$ as a union of countably many sufficiently small sets and show that each one is $m$-graph rectifiable.
	Let $\{x_i\}$ be a countable dense subset of $F$. 
	It follows from (\ref{eq:empty_cone}) and the containment $F\subset B(0, r/2)$ that for each $x_i$ there exists $F_i\subset F$ such that \[F_i\cap C_\calB(x_i,r,V,\alpha)=F_i\cap C_\calB(x_i, V, \alpha)=\emptyset\] and $\mu(F\setminus F_i)=0.$
	Define $F':=\bigcap_{i=1}^\infty F_i$.
	Then \[\mu\left(F\setminus F'\right)=\mu\left(F\setminus \bigcap_{i=1}^\infty F_i\right)=\mu\left(\bigcup_{i=1}^\infty F\setminus F_i\right)\le\sum_{i=1}^\infty\mu\left(F\setminus F_i\right)=0.\]
	
	We claim that $F'\cap C_\calB(x, V, \alpha)=\emptyset$ for every $x\in F'$.
	Fix $x\in F'$, and let $y\in C_\calB(x, V, \alpha)$.
	By definition of bad cone we have that $\dist(y-x,V)>\alpha|y-x|$.
	Now let $\epsilon>0$ such that $\dist(y-x, V)\ge \alpha(|y-x|+\epsilon)$.		
	Recalling that $0<\alpha<1$, choose $x_i$ such that $|x_i-x|<\alpha\epsilon/2<\epsilon/2$.   
	Then 
	\begin{align*}
	\dist(y-x_i, V) &\ge\dist(y-x, V)-\left|x-x_i\right|\\
	&\ge\alpha(|y-x|+\epsilon)-\alpha(\epsilon/2)\\
	&=\alpha(|y-x|+\epsilon/2)\\
	&>\alpha(|y-x|+|x_i-x|)\\
	&\ge \alpha(|y-x_i|).
	\end{align*}
	In particular, we conclude that $y\in C_\calB(x_i, V, \alpha)$. 
	Since $F_i\cap C_\calB(x_i, V,\alpha)=\emptyset$, it must be that case that $y\notin F_i$. It follows that $y\notin F'$, and thus $F'\cap C_\calB(x, V, \alpha)=\emptyset$ for all $x\in F'$.  By an application of Theorem \ref{thm:geometric_lemma} we conclude that there exists an $m$-Lipschitz graph $\Gamma$ such that $F'\subset \Gamma$, so  $\mu(F\setminus \Gamma)=0$.	
\end{proof}

\begin{lemma}
	\label{lemma:Constant_depending_on_alpha}
	Let $x\in H$, $\alpha\in (0,1)$, and $V$ be an $m$-dimensional linear plane.  If $y\in C_\calB\left(x,V,\alpha+\frac{1-\alpha}{2}\right)$ then there exists some constant $\eta_\alpha$ depending on at most $\alpha$ and the dimension of the space, $n$, such that $B(y, \eta_\alpha d)\subset C_\calB(x, V, \alpha)$ where $d:=|x-y|$.
\end{lemma}

A proof of Lemma \ref{lemma:Constant_depending_on_alpha} can be found in the \hyperref[section:Appendix]{appendix}. With the above results established, we now prove a lemma that forms the central argument for the proof of the sufficient condition of Theorem \ref{thm:Graph_rectifiable}.

\begin{lemma}\label{lemma:TheoremA_Reduction}
	Let $\mu$ be a Radon measure on $H$.  For $x_0\in H$, $V$ an $m$-dimensional linear plane, $\alpha\in(0,1)$, and parameter $K>0$, let $E$ denote the set of points $x\in H$ such that 
	\begin{enumerate}
	\item[(i)] The sequence of functions \[f_r(x):=\frac{\mu(C_\calB(x,r,V, \alpha))}{\mu(B(x,r))}\]
	converges to $0$ uniformly on $E$, and
	\item[(ii)] there exists $r_1>0$ such that at every $x\in E$, 
	\[\mu(B(x, 2r))\le K\mu(B(x, r))\text{ for all }r\in (0,r_1].\]
	\end{enumerate}
	Then $E$ is $\mu$-carried by $m$-Lipschitz graphs with Lipschitz constants depending on at most $K$ and $\alpha$.
\end{lemma}

\begin{proof}
	Fix $\delta>0$.  By uniform convergence, choose $r_\delta\le r_1$ such that for all $r<r_\delta$ and for all $x\in E$,
	\begin{equation}\label{eq:density_ratio_bound}\frac{\mu(C_\calB(x,2r,V,\alpha))}{\mu(B(x,2r))}<\delta.\end{equation}
	Fix $x\in E$, and define $S:=E\cap C_\calB(x,r, V, 2\alpha)$.
	Assuming the set is non-empty, fix $y_0\in S$ such that $|x-y_0|=\max_{y\in S}|a-y|=:\lambda r$ for some $0<\lambda\le 1$.
	As an application of Lemma \ref{lemma:Constant_depending_on_alpha} choose $\eta_\alpha$ such that $B(y_0, \eta_\alpha\lambda r)\subset C_\calB(x,2r, V, \alpha)$.  Let $d=\log_2\left(\frac{\lambda+2}{\eta_\alpha\lambda}\right)$.  
	Then
	\[2^d\eta_\alpha\lambda r=\frac{\lambda+2}{\eta_\alpha\lambda}\eta_\alpha\lambda r=(\lambda+2)r=|x-y_0|r+2r.\]
	In particular, for the specified value of $d$, $B(x, 2r)\subset B(y_0, 2^d\eta_\alpha\lambda r)$.  Applying condition (ii) of the set $E$ at the point $y_0$ we see that
	\begin{equation}
	\label{eq:doubling_containment}
	\mu(C_\calB(x,2r,V,\alpha))\ge\mu(B(x,\eta_\alpha\lambda r))
	\ge K^{-d}\mu(B(y_0,2^d\eta_\alpha\lambda r))
	\ge K^{-d}\mu(B(x, 2r))
	\end{equation}
	Combining inequalities (\ref{eq:density_ratio_bound}) and (\ref{eq:doubling_containment}), we get the density ratio bounds
	\[\delta>\frac{\mu(C_\calB(x,2r,V, \alpha))}{\mu(B(x, 2r))}\ge K^{-d}\] for all $r<r_\delta$.
	In particular, this implies that $d>\frac{-\log(\delta)}{\log K}$.
	Equivalently, 
	\[\log\left(\frac{\lambda+2}{\eta_\alpha\lambda}\right)>\frac{-\log \delta}{\log K},\]
	so that if $\delta$ is chosen to be less than $2^{-\log K\log\left(\frac{5}{\eta_\alpha}\right)}$ then  $\lambda<\frac{1}{2}$.
	From this result we conclude that for $r<r_\delta$ and for all $y\in S$, $|x-y|<\frac{1}{2} r$.  Letting $r\downarrow 0$ we conclude that $\mu(E\cap C_\calB(x,r_\delta,V, 2\alpha))=0$.
	Thus we can apply Corollary \ref{corollary:GeometricLemma}, and we obtain the desired conclusion.
\end{proof}

With Lemma \ref{lemma:TheoremA_Reduction} established, we are ready to prove Theorem \ref{thm:Graph_rectifiable}.
\begin{proof}
	We first show the sufficient condition holds.
	To do so, we use a series of countable decompositions to reduce to a setting in which Lemma \ref{lemma:TheoremA_Reduction} can be applied.
	First we may assume that $\mu$ is a finite measure, for if $\mu$ is not finite then by separability of $H$ we may write $H$ as a countable union of closed balls of radius $1$.  It follows from our definition of pointwise doubling measures that $\mu$ is finite on each ball in the union. Then the proof proceeds as below by considering the restriction of $\mu$ to each ball.

	Choose $\{V_i\}_{i=1}^\infty$ to be a dense collection of $m$-dimensional linear planes in $H$ and $\{\alpha_j\}_{j=1}^\infty$ to be a sequence dense in $(0, 1)$. 
	For a fixed $\alpha\in (0,1)$ and $m$-dimensional linear plane $V$, we can find $\alpha_k>\alpha$ and $V_l$ such that $\|V_l-V\|<\alpha_k-\alpha$. Then we have
	$C_\calB(x, V, \alpha)\subset C_\calB(x, V_l, \alpha_k)$, so of course
	\begin{equation}
	\label{eq: Zero limit fixed l}
	\text{ if }\lim_{r\downarrow 0}\frac{\mu(C_\calB(x,r, V, \alpha))}{\mu(B(x,r))}=0\text{ then }\lim_{r\downarrow 0}\frac{\mu(C_\calB(x,r, V_l, \alpha_k))}{\mu(B(x,r))}=0.\end{equation}
	Now fix some $k$ and $l$, and let 
	\[E_{k,l}:=\left\{x\in H: \lim_{r\downarrow 0}\frac{\mu(C_\calB(x,r,V_l, \alpha_k))}{\mu(B(x,r))}=0\right\}.\]	
	 By Egorov's Theorem, choose a measurable subset $E_{k,l,t}\subset E_{k,l}$ such that $\mu(E_{k,l}\setminus E_{k,l,t})<2^{-t}$ and 
	\[f^{k,l}_r(x):=\frac{\mu(C_\calB(x,r,V_l,\alpha_k))}{\mu(B(x,r))}\] converges uniformly to zero on $E_{k,l,t}$.  Note that  $H=\bigcup_{t=1}^\infty\bigcup_{l=1}^\infty\bigcup_{k=1}^\infty E_{k,l,t}$ so it suffices to show that $E_{k,l,t}$ is graph rectifiable for fixed $k$, $l$, and $t$.
	Next, since $\mu$ is pointwise doubling, for $\mu$-a.e. $x\in E_{k,l,t}$, there exists $K_x, N_x\in\mathbb{N}$ such that $\mu(B(x,2r))\le K_x\mu(B(x,r))$ for all $0<r\le 1/N_x$.
	Define 
	\[E^{K, N}_{k,l,t}=\left\{y\in E_{k,l,t}: \mu(B(y,2r))\le K\mu(B(y,r))\textrm{ for all } 0<r\le 1/N\right\}.\]
	Then $\mu\left(E_{k,l, t}\setminus\bigcup_{K=1}^\infty\bigcup_{N=1}^\infty E^{K,N}_{k,l,t} \right)=0.$ 
	Finally for $\mu$-a.e. $x\in E^{K,N}_{k,l,t}$, 
	\[\lim_{r\downarrow 0}\frac{\mu(E^{K,N}_{k, l, t}\cap B(x, r))}{\mu(B(x, r))}=1.\]
	Define
	\[E^{K,N}_{k,l,t,p}=\left\{x\in E^{K,N}_{k,l,t}:\mu(E^{K,N}_{k,l,t}\cap B(x,r))\ge \frac{1}{2}\mu(B(x,r))\text{for all } 0\le r\le 1/p\right\},\]
	and note that $E^{K,N}_{k,l,t}=\bigcup_{p=1}^\infty E^{K,N}_{k,l,t,p}$.
	To conclude the proof, apply Lemma \ref{lemma:TheoremA_Reduction} for some fixed $k$, $l$, $t$, $K$, $N$ and $p$.

	To show the necessary condition, suppose that $\mu$ is $m$-Lipschitz graph rectifiable, and let $\{\Gamma_i\}$ denote a collection of Lipschitz graphs that carry $\mu$.  To each graph $\Gamma_i$ we associate an $m$-plane $V_i$ and a number $\alpha_i\in(0,1)$ such that $\Gamma_i$ is a Lipschitz graph with respect to $V_i$ and $\alpha_i$.
	Let $x\in H$ be a $\mu$-density point. 
	Since each graph $\Gamma_i$ is closed, $x\in\Gamma_i$ for some $i$.
	It follows that
	\[\lim_{r\downarrow 0}\frac{\mu(B(x, r)\setminus \Gamma_i)}{\mu(B(x,r))}=0.\]
	Furthermore, $\Gamma_i\subset C_\calG(x, V_i, \alpha_i)$, and so $C_\calB(x, r, V_i, \alpha_i)\subset B(x,r)\setminus \Gamma_i$. It follows immediately that 
	\[\lim_{r\downarrow 0}\frac{\mu( C_\calB(x,r,V_i, \alpha_i))}{\mu(B(x,r))}=0.\]
	This completes the proof of the necessary condition.
\end{proof}


\appendix 
\section{} 
\label{section:Appendix}
In this section we collect the proofs of some results that are used a above.  The proofs are included here for completeness of the exposition.

\begin{proof}[Proof of Lemma \ref{lem:doubling_measure_satisfies_finite_overlap}]
	Let $j\ge k$.  Let $B=B(x, \lambda_22^{-k})$ and  $B'=B(y, \lambda_22^{-j})$.  Suppose that $B\cap B'\ne\emptyset$.  
	Fix $z\in B'$, then 
	\[\dist(z, x)\le \dist(z, y)+\dist(y, x)\le 2\lambda_22^{-k}.\]
	In particular, $B'\subset 2B$.
	Furthermore, for $z\in 2B$, 
	\[\dist(z, y)\le \dist(z, x)+\dist(x, y)\le 2\lambda_22^{-k}+\lambda_22^{-j}< 4\lambda_22^{-k},\]
	so $2B\subset 4\cdot 2^{j-k}B'$.
	Let $C$ denote $\#\{B_i'\in\calC^\mu_j: B'_i\cap B\ne\emptyset\}$.  Then
	\begin{equation}\begin{split}\mu(2B)\ge\mu\left(\bigcup_{i=1}^C B_i'\right)
	\ge\mu\left(\bigcup_{i=1}^C\frac{1}{2\lambda_2}B_i'\right)
	=\sum_{i=1}^C\mu\left(\frac{1}{2\lambda_2}B_i'\right)\ge\sum_{i=1}^C D^{-(j-k+3+{\log(\lambda_2)})}\mu(2^{j-k+2}B_i')\\\ge C\cdot D^{-(j-k+3+{\log(\lambda_2)})}\mu(2B).\end{split}\end{equation}
	This implies $D^{j-k+3+\log(\lambda_2)}\ge C$.  Thus we may take the finite overlap constant to be $P_{j-k}^{\mu}=D^{j-k+2+\log(\lambda_2)}$.
\end{proof}

\begin{proof}[Proof of Lemma \ref{lemma:WeightedSumsOfBalls}]
	First note that for each $k$, \[\int_{E_k}\omega(x)d\nu=\sum_{j=k}^\infty\int_{E_j\setminus E_{j+1}}\omega(x)d\nu.\]
	Therefore, 
	\begin{align*}
	\sum_{k=0}^\infty c_k\int_{E_k}\omega(x)d\nu&=\sum_{k=0}^\infty c_k\sum_{j=k}^\infty\int_{E_j\setminus E_{j+1}}\omega(x) d\nu
	=\sum_{j=0}^\infty \sum_{k=1}^j c_k\int_{E_j\setminus E_{j+1}}\omega(x)d\nu\\
	&=\sum_{j=0}^\infty\int_{E_j\setminus E_{j+1}}\sum_{k=1}^jc_k\omega(x)d\nu
	\le\sum_{j=1}^\infty\int_{E_j\setminus E_{j+1}} Cd\nu\\
	&=\sum_{j=1}^\infty C\mu(E_j\setminus E_{j+1})
	\le C\mu(E_0\setminus E).
	\end{align*}
\end{proof}

	\begin{proof}
		[Proof of Lemma \ref{lemma:lipschitz_does_not_increase_dimension}]
		Assume $P^1(E)<\infty$ and that $f:E\rightarrow H$ is $L$-Lipschitz.
		Given $\epsilon>0$, pick $\eta>0$ so that $P_\eta^1(E)\le P^1(E)+\epsilon$.
		Fix $0<\delta\le L\eta$ and let $\{B_H(f(x_i), r_i):i\le 1\}$ be an arbitrary disjoint collection of balls in $H$ centered in $f(E)$ such that $2r_i\le\delta$ for all $i\ge 1$. 
		Since $f$ is $L$-Lipschitz,
		\[f(B_\mathbb{R}(x_i, r_i/L))\subset B_H(f(x_i), r_i)\text{ for all }i\ge 1.\]
		Thus $\{B_\mathbb{R}(x_i, r_i/L):i\ge 1\}$ is a disjoint collection of balls in $\rr$ centered in $E$ such that 
		\[2r_i/L\le\delta/L\le\eta.\]
		Hence 
		\[\sum_{i=1}^\infty (2r_i)=L\sum_{i=1}^\infty (2r_i/L)\le L\cdot P^1_\eta(E)\le L(P^1(E)+\epsilon).\]
		Taking the supremum over all $\delta$ packings of $f(E)$ we obtain $P_\delta^1(f(E))\le L(P^1(E)+\epsilon)$.
		The corresponding inequality for packing measure $\mathcal{P}^1(E)$ follows immediately.
	\end{proof}
	
	\begin{proof}[Proof of Lemma \ref{lemma:density_bound_implies_measure_bound}]
		Let $E\subset A$ and $\epsilon>0$.  By definition of packing measure, choose $\delta>0$ such that $P_\delta^1(E)\le P^1(E)+\epsilon$. Using the bounded lower density assumption on $\mu$, for each $x\in E$ we can choose a sequence $\{r_{x,i}\}_{i=1}^\infty$ with $r_{x,i}\le\min\{\delta,r_0/8\}$ with   $r_{x,i}\rightarrow 0$ as $i\rightarrow\infty$ such that for each $i$, 
		\[\mu(B(x, r_{x,i}))\le\lambda(2r_{x,i}).\]
		Let $\mathcal{B}=\left\{B(x, r_{x,i}): x\in E\right\}$ where $B$ is a closed ball.  By Vitali Covering Theorem (see \cite[Theorem 2.2 and Remark 2.3 (b)]{M95}  and  or \cite[Theorem 1.6]{H01}  for results on doubling measures that can easily be adapted to the current assumptions), we can choose a subcollection $\mathcal{B}'\subset\mathcal{B}$ such elements of $\mathcal{B}'$ are disjoint and 
		\[\mu\left(E\setminus\bigcup_{B\in\mathcal{B}'} B\right)=0.\]
		
		Then 
		\[\mu(E)\le\sum_{i=1}^\infty \mu(B_i)\\
		\le\sum_{i=1}^\infty \lambda 2r_{B_i}\\
		\le P_\delta^1(E)
		\le\lambda(P^1(E)+\epsilon).\]
		
		Let $\epsilon\downarrow 0$ to conclude $\mu(E)\le\lambda P^1(E)$ for $E\subset A$.
		Thus, for $A=\bigcup_{l=1}^\infty E_l$, 
		\[\mu(A)\le\sum_{l=1}^\infty\mu(E_l)\le \lambda\sum_{l=1}^\infty P^1(E_l).\]
		Hence $\mu(A)\le\lambda\mathcal{P}^1(A)$.
	\end{proof}	

The proof of Lemma \ref{lemma:convergence_to_compact set} relies on fundamental properties of excess and Hausdorff distance.  
For nonempty sets $S, T\subset X$, the excess, $ex(S, T)$ of $S$ over $T$ is defined by 
\[ex(S, T):=\sup_{s\in S}\inf_{t\in T}\dist(s,t)\]
and the Hausdorff distance $HD(S, T)$ between $S$ and $T$ is defined by 
\[HD(S, T):=\max\{ex(S,T), ex(T,S)\}.\]
Let $CL(H)$ denote the set of nonempty closed subsets of $H$.  
Since $(H,|\cdot|)$ is a complete metric space, $(CL(H), HD)$ is also a complete metric space.
See (\cite{B93}, Chapter 3) for details.


\begin{proof}[Proof of Lemma \ref{lemma:convergence_to_compact set}]
	Let $n\ge 1$, $C^*>1$, $\delta\le 1/2$ and $r_0>0$.  
	Assume that $V_0, V_1, V_2,...$ is a sequence of nonempty, closed finite subsets of a bounded set $B$ such that each $V_i$ satisfies (V2) and (V3). 
	By iterating (V2), we obtain that for any $k<j$ and $v_k\in V_k$, we can find a sequence of $v_i\in V_i$, $i=k+1,...,j$ such that
	\[|v_k-v_j|\le |v_k-v_{k+1}|+...+|v_{j-1}-v_j|< C^*\delta^{k}r_0+...+C^*\delta^{j-1}r_0\le 2C^*\delta^{k}.\] 
	It follows that 
	\[ex(V_k, V_j)<3C^*\delta^{k}r_0.\]
	Similarly iterating (V3), we obtain that for any $k<j$,
	\[ex(V_j, V_k)<3C^*\delta^{k}r_0.\]
	Thus $HD(V_k, V_j)\le 2C^*\delta^{k}r_0$.  
	In particular this implies that $\{V_k\}$ is a Cauchy sequence of sets.
	By the completeness of $(CL(H), HD)$, $\{V_k\}$ converges to a closed set $V$.\qedhere
\end{proof}

\begin{proof}[Proof of Theorem \ref{thm:geometric_lemma}]
	Let $x\in F$.  
	Let $P_V:H\rightarrow V$ denote standard projection onto the $m$-plane $V$.  
	Suppose that 
	$|P_{V}x-P_{V}y|<(1-\alpha^2)^{1/2}|x-y|$.  
	Then $y\in C_\calB(x, V, \alpha)$, and by assumption of $F$ this means that $y\notin F$.
	Thus we may assume that if $x, y\in F$ then 
	\[|P_{V}x-P_{V}y|\ge (1-\alpha^2)^{1/2}|x-y|.\]
	From this inequality we see that $P_{V}|F$ is one-to-one with Lipschitz inverse $f=(P_{V}|F)^{-1}$ and $Lip(f)\le (1-\alpha^2)^{-1/2}$.  
	Note that $F=f(P_V|F)$.  Then there exists a Lipschitz extension $\tilde{f}: V\rightarrow H$ so that $F\subset \tilde{f}(V)$.  Thus the desired result holds.
\end{proof}

\begin{proof}[Proof of Lemma \ref{lemma:Constant_depending_on_alpha}]
	To determine the maximum constant $\eta_\alpha$, we consider a point $b\in\partial C_\calB\left(x, V, \alpha+\frac{1-\alpha}{2}\right)$ and determine the distance $d'$ from $b$ to $C_\calG(x, V,\alpha)$.	Define $\theta$ to be the angle between $\partial C_\calB\left(x, V, \alpha+\frac{1-\alpha}{2}\right)$, and $\partial C_\calB(x, V, \alpha)$, $\theta'$ to be the angle between $\partial C_\calB(x, V, \alpha)$ and $V$ and $\theta''$ to be the angle between $\partial C_\calB(a, V^\perp, \alpha+\frac{1-\alpha}{2})$ and $V$.
	Note that  $\theta=\theta''-\theta'$.
	Some simple calculations show that $\theta''=cos^{-1}((\alpha+\frac{1-\alpha}{2})')$, and $\theta'=cos^{-1}((\alpha)')$.
	Thus $\theta=cos^{-1}((\alpha+\frac{1-\alpha}{2})')-\cos^{-1}(1-\alpha)$.
	Again simple calculations show that \[d'=d\sin\left(\cos^{-1}\left(\left(\alpha+\frac{1-\alpha}{2}\right)'\right)-\cos^{-1}((\alpha)')\right)\]
	Letting $t_1=1-\alpha$ and $t_2=1-2\alpha$, we rewrite
	\[d'=d\left(\sqrt{1-\left(t_2t_1+\sqrt{(1-t_2^2)(1-t_1^2)}\right)}\right)=:d\eta_\alpha.\qedhere\]
\end{proof}

\bibliography{Rectifiable_Measures_Hilbert_Space_References} 
\bibliographystyle{amsbeta}

\end{document}